\documentclass[12pt]{amsart}
\usepackage{amsmath}
\usepackage{amssymb}
\usepackage[all]{xy}
\usepackage{longtable}
\usepackage[mathscr]{eucal}
\usepackage{xcolor}
\usepackage[bookmarks=false]{hyperref}
\usepackage{multirow,bigdelim}
\usepackage{nccmath}
\usepackage{url}
\usepackage{ulem}
\usepackage{comment}
\usepackage{mathrsfs}
\usepackage{arydshln}
\usepackage{stmaryrd}
\usepackage{dsfont}
\usepackage{bm}
\usepackage{tikz}
\usepackage{tikz-cd}

\newtheorem{theorem}[equation]{Theorem}
\newtheorem{lemma}[equation]{Lemma}
\newtheorem{proposition}[equation]{Proposition}
\newtheorem{corollary}[equation]{Corollary}

\newtheorem{definition}[equation]{Definition}

\newtheorem{theorem-n}{Theorem}
\newtheorem{claim-n}[theorem-n]{Claim}
\newtheorem{lemma-n}[theorem-n]{Lemma}
\newtheorem{proposition-n}[theorem-n]{Proposition}

\theoremstyle{definition}
\newtheorem{definition-n}[theorem-n]{Definition}
\newtheorem{example}[equation]{Example}

\theoremstyle{remark}
\newtheorem{remark}[equation]{Remark}
\newtheorem{remark-n}[theorem-n]{Remark}

\numberwithin{equation}{subsection}

\allowdisplaybreaks[1]

\newcommand{\FF}{\mathbb{F}}

\newcommand{\TT}{\mathbb{T}}

\newcommand{\EE}{\mathbb{E}}
\newcommand{\CC}{\mathbb{C}}

\newcommand{\KK}{\mathbb{K}}

\newcommand{\bu}{\mathbf{u}}
\newcommand{\bv}{\mathbf{v}}

\DeclareMathAlphabet{\matheur}{U}{eur}{m}{n}

\newcommand{\fs}{\mathfrak{s}}

\newcommand{\ff}{\mathfrak{f}}

\newcommand{\fn}{\mathfrak{n}}

\DeclareMathOperator{\Ker}{Ker} \DeclareMathOperator{\GL}{GL}
\DeclareMathOperator{\Mat}{Mat} 
\DeclareMathOperator{\End}{End} 
 
\DeclareMathOperator{\Hom}{Hom} 
 
\DeclareMathOperator{\Id}{Id} 
 
\DeclareMathOperator{\Ext}{Ext}  
\DeclareMathOperator{\Li}{Li}
\DeclareMathOperator{\rank}{rank}

\newcommand{\ok}{\overline{k}}
\newcommand{\oK}{\overline{K}}

\newcommand{\sep}{\mathrm{sep}}
\newcommand{\tr}{\mathrm{tr}}

\newcommand{\laurent}[2]{{#1 (\!( #2 )\!)}}

\definecolor{ForestGreen}{rgb}{0.0, 0.5, 0.0}

\makeatletter
\newcommand{\xequal}[2][]{\ext@arrow 0055{\equalfill@}{#1}{#2}}
\def\equalfill@{\arrowfill@\Relbar\Relbar\Relbar}
\makeatother

\title [Analytic continuation of Kochubei multiple polylogarithms]{Analytic continuation of Kochubei multiple polylogarithms and its applications}

\author{Yen-Tsung Chen}
\address{Department of Mathematics, Penn State University, University Park, PA 16801, USA}
\email{ytchen.math@gmail.com}

\date{\today}

\begin{document}

\begin{abstract}
    In the present paper, we propose an analytic continuation of Kochubei multiple polylogarithms by using the techniques developed in \cite{Fur22}. 
    Moreover, we produce a family of linear relations and a linear independence result for values of our analytically continued Kochubei polylogarithms at algebraic elements from a cohomological aspect.
\end{abstract}

\maketitle

\tableofcontents

\section{Introduction}
\subsection{Classical multiple polylogarithms}
    Let $n\geq 1$ be an integer. The classical $n$-th polylogarithm is the power series given by
    \[
        \Li_n(z):=\sum_{m\geq 1}\frac{z^m}{m^n}\in\mathbb{Q}\llbracket z\rrbracket.
    \]
    It defines a complex-valued function on the open unit disc $\{z\in\mathbb{C}\mid |z|<1\}$ inside the complex plane. It satisfies the differential equations
    \begin{equation}\label{Eq:PolyDE_geq2}
        z\frac{d}{dz}\Li_n(z)=\Li_{n-1}(z),~\mathrm{if}~n\geq 2,
    \end{equation}
    and
    \begin{equation}\label{Eq:PolyDE_eq1}
        (1-z)\frac{d}{dz}\Li_1(z)=1.
    \end{equation}
    By using iterated path integrals, the $n$-th polylogarithm can be analytically continued to a multi-valued meromorphic function on $\mathbb{C}$. Let $\fs=(s_1,\dots,s_r)\in\mathbb{Z}_{>0}^r$. Then the above functions can be generalized to the $\fs$-th multiple polylogarithm. More precisely, the following power series in several variables
    \[
        \Li_\fs(z_1,\dots,z_r):=\sum_{m_1>\cdots>m_r\geq 1}\frac{z_1^{m_1}\cdots z_r^{m_r}}{m_1^{s_1}\cdots m_r^{s_r}}\in\mathbb{Q}\llbracket z_1,\dots,z_r\rrbracket
    \]
    defines a complex-valued function on the open unit polydisc $\{(z_1,\dots,z_r)\in\mathbb{C}^r\mid |z_i|<1,~1\leq i\leq r\}$. It also satisfies a generalization of \eqref{Eq:PolyDE_geq2} and $\eqref{Eq:PolyDE_eq1}$ (see for example \cite[Sec.~3]{Wal02}). In particular, it can be also analytically continued to a larger region by iterated integrals (cf. \cite{Zha07}).

    The main theme of the present article is to introduce an analytic continuation of Kochubei multiple polylogarithms, a variant of multiple polylogarithms over function fields in positive characteristic, and determine its monodromy module in the sense of Furusho \cite{Fur22}.

\subsection{Main results}
    Let $\mathbb{F}_q$ be the finite field of $q$ elements where $q=p^\ell$ for some positive integer $\ell$ and prime number $p$. We denote by $A=\mathbb{F}_q[\theta]$ the polynomial ring in variable $\theta$ over $\mathbb{F}_q$ and $K=\mathbb{F}_q(\theta)$ the fraction field of $A$. The field $K$ is equipped with a normalized non-Archimedean norm $|\cdot|_\infty$ so that $|f/g|_\infty:=q^{\deg_\theta f-\deg_\theta g}$ where $f,g\in A$ with $g\neq 0$. Let $K_\infty$ be the completion of $K$ with respect to $|\cdot|_\infty$. We identify $K_\infty$ with the Laurent series field $\laurent{\FF_q}{1/\theta}$ in $1/\theta$ over $\mathbb{F}_q$. Finally, we set $\mathbb{C}_\infty$ to be the completion of a fixed algebraic closure of $K_\infty$ and set $\oK$ to be the algebraic closure of $K$ inside $\mathbb{C}_\infty$.

    In \cite{Koc05}, Kochubei introduced an analogue of the classical polylogarithms over function fields in positive characteristic using a system of difference equations. To be more specific, consider the twisted power series ring $\mathbb{C}_\infty\llbracket \tau\rrbracket$ subject to the following relation $\alpha^q\tau=\tau\alpha$ for each $\alpha\in\mathbb{C}_\infty$. We define the Carlitz difference operator on $\mathbb{C}_\infty\llbracket \tau\rrbracket$ by setting
    \begin{align*}
        \Delta:\mathbb{C}_\infty\llbracket \tau\rrbracket&\to\mathbb{C}_\infty\llbracket \tau\rrbracket\\
        F(\tau)&\mapsto(\Delta F)(\tau):=F(\tau)\theta-\theta F(\tau).
    \end{align*}
    It is straightforward to see that $\Delta(\mathbb{C}_\infty\llbracket \tau\rrbracket)\subset\mathbb{C}_\infty\llbracket \tau\rrbracket\tau$. Thus, we have a well-defined operator
    \begin{align*}
        d_C:=\tau^{-1}\Delta:\mathbb{C}_\infty\llbracket \tau\rrbracket&\to\mathbb{C}_\infty\llbracket \tau\rrbracket\\
        F(\tau)&\mapsto(d_CF)(\tau):=\tau^{-1}\big(F(\tau)\theta-\theta F(\tau)\big).
    \end{align*}
    We set $D_0:=1$ and for each $i\in\mathbb{Z}_{\geq 1}$ we define $D_i:=\prod_{j=1}^i(\theta^{q^j}-\theta)^{q^{i-j}}$. The power series $\exp_C(\tau):=\sum_{i\geq 0}\frac{1}{D_i}\tau^i\in\mathbb{C}_\infty\llbracket\tau\rrbracket$ is the Carlitz exponential series, and can be regarded as the function field analogue of the classical exponential series $\exp(z)=\sum_{m\geq 0}\frac{1}{m!}z^m\in\mathbb{C}\llbracket z\rrbracket$. Since $\exp_C(\tau)$ satisfies the relation
    \[
        (\theta+\tau)\exp_C(\tau)=\exp_C(\tau)\theta,
    \]
    we conclude immediately that
    \[
        d_C\big(\exp_C(\tau)\big)=\exp_C(\tau).
    \]
    Thus, if we make the following correspondence
    \[
        \frac{d}{dz}\leftrightarrow d_C,~z\frac{d}{z}\leftrightarrow\Delta=\tau d_C,~(1-z)\frac{d}{dz}\leftrightarrow(1-\tau)d_C,
    \]
    then \eqref{Eq:PolyDE_geq2} and \eqref{Eq:PolyDE_eq1} can be transformed to the following difference equations with indeterminate series $F_n(\tau)\in\mathbb{C}_\infty\llbracket\tau\rrbracket$ for each $n\geq 1$, namely,
    \begin{equation}\label{Eq:KPLDE_geq2}
        (\Delta F_n)(\tau)=F_{n-1}(\tau),~\mathrm{if}~n\geq 2,
    \end{equation}
    and
    \begin{equation}\label{Eq:KPLDE_eq1}
        (1-\tau)(d_CF_1)(\tau)=1.
    \end{equation}
    In fact, if we introduce $F_0(\tau):=\sum_{i\geq 1}\tau^i$, then \eqref{Eq:KPLDE_eq1} can be realized as
    \[
        (\Delta F_1)(\tau)=F_0(\tau).
    \]

    Note that for each $F(\tau)=\sum_{i\geq 0}a_i \tau^i\in\mathbb{C}_\infty\llbracket\tau\rrbracket$, we have an induced $\mathbb{F}_q$-linear power series $F(z)=\sum_{i\geq 0}a_iz^{q^i}\in\mathbb{C}_\infty\llbracket z\rrbracket$. In addition, the action of the Carlitz difference operator $\Delta$ on such a series $F(z)$ can be understood as
    \[
        (\Delta F)(z)=F(\theta z)-\theta F(z)\in\mathbb{C}_\infty\llbracket z\rrbracket.
    \]
    For $n\in\mathbb{Z}_{\geq 0}$, the $n$-th Kochubei polylogarithm, abbreviated as KPL, is given by the following $\mathbb{F}_q$-linear power series
    \begin{equation}\label{Eq:KPLs_Def}
        \Li_{\mathcal{K},n}(z):=\sum_{i\geq 1}\frac{z^{q^i}}{(\theta^{q^i}-\theta)^n}\in\mathbb{C}_\infty\llbracket z\rrbracket.
    \end{equation}
    It defines an analytic function on $\mathbb{C}_\infty$ converging at $z\in\mathbb{C}_\infty$ with $|z|_\infty<q^n$.
    One checks directly that for $n\geq 1$ we have
    \begin{equation}\label{Eq:DifferenceEquation_n>1}
        (\Delta\Li_{\mathcal{K},n})(z)=\Li_{\mathcal{K},n}(\theta z)-\theta\Li_{\mathcal{K},n}(z)=\sum_{i\geq 1}\frac{(\theta^{q^i}-\theta)z^{q^i}}{(\theta^{q^i}-\theta)^n}=\Li_{\mathcal{K},n-1}(z).
    \end{equation}
    In other words, $\Li_{\mathcal{K},n}(z)$ are solutions of \eqref{Eq:KPLDE_geq2} and \eqref{Eq:KPLDE_eq1}, and can be regarded as an analogue of the classical polylogarithms.

    Let $\fs=(s_1,\dots,s_r)\in\mathbb{Z}_{\geq 0}^r$. The $\fs$-th Kochubei multiple polylogarithm, abbreviated as KMPL, was introduced by Harada in \cite{Har22} as a generalization of KPLs. Note that the KMPLs introduced in \cite{Har22} only allowed $\fs\in\mathbb{Z}_{>0}^r$, but here we include the cases where some of $s_i$ equal to zero. This is just a convention which simplifies some of our formulation later. The $\fs$-th KMPL is given by the following power series in several variables
    \begin{equation}\label{Eq:KMPLs_Def}
        \Li_{\mathcal{K},\fs}(z_1,\dots,z_r):=\sum_{i_1>\cdots>i_r>0}\frac{z_1^{q^{i_1}}\cdots z_r^{q^{i_r}}}{(\theta^{q^{i_1}}-\theta)^{s_1}\cdots(\theta^{q^{i_r}}-\theta)^{s_r}}\in\mathbb{C}_\infty\llbracket z_1,\dots,z_r\rrbracket.
    \end{equation}
    Note that $\Li_{\mathcal{K},\fs}(z_1,\dots,z_r)$ converges on
    \begin{equation}\label{Eq:D_fs}
        \mathbb{D}_\fs:=\{(z_1,\dots,z_r)\in\mathbb{C}_\infty^r\mid|z_1/\theta^{s_1}|^{q^{i_1}}_\infty\cdots|z_1/\theta^{s_r}|^{q^{i_r}}_\infty\to 0,~\mathrm{as}~0<i_r<\cdots<i_1\to\infty\}.
    \end{equation}
    In particular, it converges on the open polydisc
    \[
        \mathbb{D}_\fs':=\{(z_1,\dots,z_r)\in\mathbb{C}_\infty^r\mid|z_i|_\infty<q^{s_i},~1\leq i\leq r\}\subset\mathbb{D}_\fs.
    \]
    Moreover, if we fix $u_2,\dots,u_r\in\mathbb{C}_\infty$ with $|u_i|_\infty<q^{s_i}$, then $\Li_{\mathcal{K},\fs}(z_1,u_2,\dots,u_r)\in\mathbb{C}_\infty\llbracket z_1\rrbracket$ is a power series in one variable $z_1$.
    One checks directly that for $s_1\geq 1$ we have
    \begin{equation}\label{Eq:DifferenceEquation_KMPLs}
        (\Delta\Li_{\mathcal{K},\fs})(z_1,u_2,\dots,u_r)=\Li_{\mathcal{K},(s_1-1,s_2,\dots,s_r)}(z_1,u_2,\dots,u_r).
    \end{equation}
    This is a higher-dimensional generalization of \eqref{Eq:DifferenceEquation_n>1} and can be treated as an analogue of the differential equations satisfied by classical multiple polylogarithms.

    Due to the lack of the theory of iterated integrals in positive characteristics, it is not immediate to see how to carry out analytic continuation for functions in our equal-characteristic setting. However, using the theory of the Artin-Schreier equations, Furusho \cite{Fur22} proposed an analytic continuation method which can be viewed as a substitute for iterated path integrals. Then Furusho established in \cite{Fur22} an analytic continuation of Carlitz multiple polylogarithms, abbreviated as CMPLs. Here CMPLs are another variants of the classical multiple polylogarithms over function fields in positive characteristic, where the Carlitz logarithm was initiated by Carlitz, then generalized to Carlitz polylogarithms by Anderson and Thakur in \cite{AT90}, and finally extended to CMPLs by Chang in \cite{Cha14}. The same technique can be also applied to construct an analytic continuation of the Drinfeld logarithms \cite{Che24}.

    The first main result in this article is an analytic continuation of KMPLs using Furusho's method. Our analytically continued KMPLs satisfy the same difference equation \eqref{Eq:DifferenceEquation_KMPLs} as the original ones. To be more specific, for an $A$-module $W$, we denote by $\Hom_{\mathbb{F}_q}(\mathbb{C}_\infty,W)$ the set of $\mathbb{F}_q$-linear functions from $\mathbb{C}_\infty$ into $W$. For our purposes, the $A$-module $W$ will be the quotients of the $A$-module $\Mat_{r\times 1}(\mathbb{C}_\infty)$ by some of its $A$-submodules, where the $A$-module structure on $\Mat_{r\times 1}(\mathbb{C}_\infty)$ is simply the scalar multiplication, namely
    \[
        a\cdot(z_1,\dots,z_r)^\tr:=(az_1,\dots,az_r)^\tr\in\Mat_{r\times 1}(\mathbb{C}_\infty).
    \]
    By abuse of the notation, the Carlitz difference operator $\Delta$ on $F\in\Hom_{\mathbb{F}_q}(\mathbb{C}_\infty,W)$ is defined by
    \[
        (\Delta F)(z):=F(\theta z)-\theta F(z).
    \]
    We establish the following analytic continuation of KMPLs, which will be restated in Theorem~\ref{Thm:KPLs} and Theorem~\ref{Thm:KMPLs} later.
    \begin{theorem}\label{Thm:Intro_I}
        Let $\fs=(s_1,\Tilde{\fs})\in\mathbb{Z}_{\geq 0}^r$ with $\Tilde{\fs}=(s_2,\dots,s_r)\in\mathbb{Z}_{\geq 0}^{r-1}$ and $\Tilde{\bu}=(u_2,\dots,u_r)\in\Mat_{1\times(r-1)}(\mathbb{C}_\infty)$. There exist a discrete $A$-submodule $\mathcal{M}_{\widetilde{\bu},\widetilde{\fs}}\subset\Mat_{r\times 1}(\mathbb{C}_\infty)$ and an $\mathbb{F}_q$-linear function 
        \[
            \overset{\rightarrow}{\Li_{\mathcal{K},\fs}}(-,u_2,\dots,u_r)\in\Hom_{\mathbb{F}_q}\big(\mathbb{C}_\infty,\Mat_{r\times 1}(\mathbb{C}_\infty)/\mathcal{M}_{\widetilde{\bu},\widetilde{\fs}}\big)
        \]
        such that the following assertions hold.
        \begin{enumerate}
            \item For $u_1\in\mathbb{C}_\infty$ with $\bu=(u_1,\dots,u_r)\in\mathbb{D}_\fs$, we have
            \[
                \overset{\rightarrow}{\Li_{\mathcal{K},\fs}}(u_1,\dots,u_r)\equiv\begin{pmatrix}
                    \Li_{\mathcal{K},s_1}(u_1)\\
                    \Li_{\mathcal{K},(s_1,s_2)}(u_1,u_2)\\
                    \vdots\\
                    \Li_{\mathcal{K},\fs}(\bu)
                \end{pmatrix}~(\mathrm{mod}~\mathcal{M}_{\widetilde{\bu},\widetilde{\fs}}).
            \]
            \item The $\mathbb{F}_q$-linear function
            \[
                \overset{\rightarrow}{\Li_{\mathcal{K},\fs}}(-,u_2,\dots,u_r):\mathbb{C}_\infty\to\Mat_{r\times 1}(\mathbb{C}_\infty)/\mathcal{M}_{\widetilde{\bu},\widetilde{\fs}}
            \]
            locally admits an analytic lift as a function on $u_1$, that is, for any point $u_1\in\mathbb{C}_\infty$, there is a closed disk $U\subset\mathbb{C}_\infty$ with center $u_1$ and an $\mathbb{F}_q$-linear lift of $\overset{\rightarrow}{\Li_{\mathcal{K},\fs}}(-,u_2,\dots,u_r)$, denoted by ${\Li_{\mathcal{K},\fs}^\circ}(-,u_2,\dots,u_r):\mathbb{C}_\infty\to\mathbb{C}_\infty$, so that ${\Li_{\mathcal{K},\fs}^\circ}(-,u_2,\dots,u_r)\mid_{U}:U\to\mathbb{C}_\infty$ is induced by a power series.
            \item For $s_1\geq 1$ and $u_1\in\mathbb{C}_\infty$, we have
            \begin{align*}
                (\Delta\overset{\rightarrow}{\Li_{\mathcal{K},\fs}})(u_1,\dots,u_r)=\overset{\rightarrow}{\Li_{\mathcal{K},(s_1-1,s_2,\dots,s_r)}}(u_1,\dots,u_r).
            \end{align*}
        \end{enumerate}
    \end{theorem}

    Note that in the case of KPLs ($r=1$), we have $\widetilde{\fs}=\emptyset$ and $\widetilde{\bu}=\emptyset$. The discrete $A$-submodule described in Theorem~\ref{Thm:Intro_I} is simply given by $\mathcal{M}_{\emptyset,\emptyset}=A\subset\mathbb{C}_\infty$. In other words, for any $n\in\mathbb{Z}_{\geq 0}$ our $n$-th analytically continued KPL is a $\mathbb{F}_q$-linear function
    \[
        \overset{\rightarrow}{\Li_{\mathcal{K},n}}(-):\mathbb{C}_\infty\to\mathbb{C}_\infty/A.
    \]
    For $r\geq 2$, if $\widetilde{\bu}\in\mathbb{D}_{\widetilde{\fs}}$, then the discrete $A$-submodule $\mathcal{M}_{\widetilde{\bu},\widetilde{\fs}}$ described in Theorem~\ref{Thm:Intro_I} is the $A$-linear span of the columns of the following matrix
    \[
        \begin{pmatrix}
            1 & 0 & \cdots & \cdots & 0\\
            \Li_{\mathcal{K},s_2}(u_2) & 1 & & & \vdots\\
            \Li_{\mathcal{K},(s_2,s_3)}(u_2,u_3) & \Li_{\mathcal{K},s_3}(u_3) & \ddots & & \vdots\\
            \vdots & & \ddots & & \vdots\\
            \vdots & & & 1 & 0\\
            \Li_{\mathcal{K},\widetilde{\fs}}(\widetilde{\bu}) & \cdots & \Li_{\mathcal{K},(s_{r-1},s_r)}(u_{r-1},u_r) & \Li_{\mathcal{K},s_r}(u_r) & 1
        \end{pmatrix}\in\GL_r(\mathbb{C}_\infty).
    \]
    
    The second main result of the present paper is a closer look at the values of our analytically continued KPLs, including an explicit series expression of $\overset{\rightarrow}{\Li_{\mathcal{K},n}}(u)$ for arbitrary $u\in\oK$ and the study of linear relations among KPLs at algebraic element from a cohomological point of view. To be more precise, for an integer $n\geq 1$, we consider the $\oK$-vector space $\mathcal{V}_n=\Mat_{n\times 1}(\oK)$ equipped with the $\mathbb{F}_q[t]$-module structure
    \begin{align*}
        \mathbb{F}_q[t]&\to\End_{\oK}(\mathcal{V}_n)\cong\Mat_{n\times n}(\oK)\\
        a&\mapsto [a]_n,
    \end{align*}
    that is determined by
    \[
        [t]_n=\begin{pmatrix}
                \theta & 1 & &\\
                 & \ddots & \ddots & \\
                 & & \ddots & 1\\
                 & & & \theta
            \end{pmatrix}.
    \]
    A cohomological meaning behind $\mathcal{V}_n$ will be presented in Lemma~\ref{Lem:TMod}. 
    
    Now we are ready to state the second main result of this article, which will be restated as Corollary~\ref{Cor:Series_Expression} and Theorem~\ref{Thm:Linear_Relations} later.

    \begin{theorem}\label{Thm:Intro_II}
        Let $n\in\mathbb{Z}_{>0}$ and $u_1,\dots,u_\ell\in\oK$. If we set
        \[
            \bv_1:=\begin{pmatrix}
                0\\
                \vdots\\
                0\\
                (-1)^nu_1
            \end{pmatrix},\dots,\bv_\ell:=\begin{pmatrix}
                0\\
                \vdots\\
                0\\
                (-1)^nu_\ell
            \end{pmatrix}\in\mathcal{V}_n,
        \]
        and fix an $\mathbb{F}_q$-linear lift $\Li_{\mathcal{K},n}^\circ:\mathbb{C}_\infty\to\mathbb{C}_\infty$ of $\overset{\rightarrow}{\Li_{\mathcal{K},n}}$, then the following assertions hold.
        \begin{enumerate}
            \item Let $u\in\oK$. There exist $m\in\mathbb{Z}_{>0}$ and  explicitly constructed $c\in\oK$, $a_1,\dots,a_m\in A$, as well as $f_1,\dots,f_m\in\oK$ with $|f_i|_\infty<q^n$ so that
            \[
                \overset{\rightarrow}{\Li_{\mathcal{K},n}}(u)=c+\sum_{j=1}^ma_j\overset{\rightarrow}{\Li_{\mathcal{K},n}}(f_i)=c+\sum_{j=1}^ma_j\left(\sum_{i\geq 1}\frac{f_j^{q^i}}{(\theta^{q^i}-\theta)^n}\right)\in\mathbb{C}_\infty/A.
            \]
            \item $\dim_K\mathrm{Span}_K\{1,\Li_{\mathcal{K},n}^\circ(u_1),\dots,\Li_{\mathcal{K},n}^\circ(u_\ell)\}\geq\rank_{\mathbb{F}_q[t]}\mathrm{Span}_{\mathbb{F}_q[t]}\{\bv_1,\dots,\bv_\ell\}+1$.
            \item $\rank_{\mathbb{F}_q[t]}\mathrm{Span}_{\mathbb{F}_q[t]}\{\bv_1,\dots,\bv_\ell\}+1\geq\dim_{\oK}\mathrm{Span}_{\oK}\{1,\Li_{\mathcal{K},n}^\circ(u_1),\dots,\Li_{\mathcal{K},n}^\circ(u_\ell)\}$.
        \end{enumerate}
    \end{theorem}

    We mention that if $u_1,\dots,u_\ell$ are $K$-linearly independent, then $\bv_1,\dots,\bv_\ell$ are always $\mathbb{F}_q[t]$-linearly independent in $\mathcal{V}_n$. Thus, Theorem~\ref{Thm:Intro_II}(2) can be used to detect more $K$-linearly independent sets of KPLs at algebraic elements than
    \cite[Thm.~4.9]{Har22} (cf. \cite[Rem.~4.10]{Har22}).
    We refer the reader to Example~\ref{Ex:Independence} for a concrete example. Moreover, Theorem~\ref{Thm:Intro_II}(3) says that if $\bv_1,\dots,\bv_\ell$ have non-trivial $\mathbb{F}_q[t]$-linear relations, then they can be lifted to non-trivial $\oK$-linear relations among $1,\Li_{\mathcal{K},n}^\circ(u_1),\dots,\Li_{\mathcal{K},n}^\circ(u_\ell)$. This is a supplementary result to the study of the linear independence given in \cite{Har22} as the topic of generating linear relations is not included there.

    \begin{remark}
        At the writing of this paper, R. Harada pointed out to the author that our analytically continued KPLs potentially give analytic continuation of some Thakur's hypergeometric functions \cite{Tha95} and Hasegawa's log-type hypergeometric functions \cite{Has22}, since they can be expressed by KPLs (cf. \cite[Prop~2.3]{Har22} and \cite[\S~4.3]{Has22}). We wish to further investigate this direction in a future work.
    \end{remark}
    
\subsection{Strategy and organization}
    The first part of this paper deals with Furusho's analytic continuation of KMPLs. It begins with two key ingredients: a slightly generalized version of Harada's $t$-deformation series $\mathscr{L}_{\fs,\bu}$ and a variant of Furusho's $\wp$-function on the Tate algebra $\TT$ of the closed unit disk in $\mathbb{C}_\infty$. To be more precise, for $\fs=(s_1,\dots,s_r)\in\mathbb{Z}_{\geq 0}^r$
    and $\bu\in\mathbb{D}_\fs$, we have $\mathscr{L}_{\fs,\bu}\in\TT$. Moreover, it is regular at $t=\theta$ and its specialization recovers the $\fs$-th KMPL at $\bu$. The process of Furusho's analytic continuation starts with the observation that $\mathscr{L}_{\fs,\bu}$ is a representative of an iterated inverse image of a sequence of rational functions $h_1,\dots,h_r\in K(t)\cap\TT$ under $\wp$. It allows us to define an extension $\overset{\rightarrow}{\mathscr{L}_{\fs,\bu}}$ for any choice of $\fs$ and $\bu$. However, it is subtle to evaluate $\overset{\rightarrow}{\mathscr{L}_{\fs,\bu}}$ at $t=\theta$ as $h_i$ has poles of order at most $s_i$ at $t=\theta$ for each $1\leq i\leq r$. This is the main difference compared to the prior analytic continuation results given in \cite{Fur22} and \cite{Che24}, since the constructions there deal only with the inverse image of entire functions under $\wp$. To deduce the regularity at $t=\theta$ for $\overset{\rightarrow}{\mathscr{L}_{\fs,\bu}}$, we relate $\overset{\rightarrow}{\mathscr{L}_{\fs,\bu}}$ to the solution space of some explicitly constructed Frobenius difference equation. Then the property established in \cite[Prop.~3.1.3]{ABP04} enables us to conclude the desired regularity at $t=\theta$.

    The second part of the present article investigates the linear relations among our analytically continued KPLs at algebraic elements. The content here is inspired by the results on the connection between linear relations among CPLs at algebraic points and the associated $\Ext^1$-module in the category of Frobenius modules (see \cite{Cha16},\cite{CPY19},\cite{CH21},\cite{GM22}). Note that the $\Ext^1$-module related to CPLs naturally has a $t$-module structure due to the theory of Anderson (see \cite{And86},\cite{Cha16},\cite{CPY19},\cite{Tae20},\cite{Gaz24}). In the setting of KPLs at algebraic points, a different $\Ext^1$-module is naturally attached. However, unlike the case of CPLs at algebraic points, the structure of $\Ext^1$-module in question can not be determined immediately from Anderson's theory. Fortunately, with the hint pointed out by Papanikolas and Ramachandran in \cite[Rem.~4.2]{PR03}, we can carry out an explicit description of the desired $\Ext^1$-module associated to KPLs. Then using the so-called ABP criterion \cite[Thm.~3.1.1]{ABP04}, a standard treatment given in \cite{CPY19} will lead to Theorem~\ref{Thm:Intro_II}.

    The organization of this paper is given below. In Section 2, we set up our notation and present essential properties of the hyperderivatives and Frobenius modules. In Section 3.1, we first illustrate the key ideas of Furusho's analytic continuation for KPLs. Then we extend our results to KMPLs in Section 3.2. In Section 4.1, we determine the structure of the desired $\Ext^1$-module. Finally, we use the results established in Section 4.1 to prove Theorem~\ref{Thm:Intro_II} in Section 4.2.

\section*{Acknowledgement}
    The author would like to thank H. Furusho, O. Gezmis, R. Harada, C. Namoijam, and F. Pellarin for fruitful comments and suggestions. The author was partially supported by the AMS-Simons Travel Grants and the Department of Mathematics at Penn State University. Parts of this work were done when the author visited National Tsing Hua University in January 2025. The author thanks C.-Y Chang for his hospitality.

\section{Preliminaries}
\subsection{Notation}
    \begin{longtable}{p{2.25truecm}@{\hspace{5pt}$=$\hspace{5pt}}p{11truecm}}
    $\FF_q$ & finite field with $q$ elements, where $q$ is a positive power of a prime $p$. \\
    $A$ & $\FF_q[\theta]$, the polynomial ring in $\theta$ over $\FF_q$. \\
    $K$ & $\FF_q(\theta)$, the fraction field of $A$. \\
    $|\cdot|_\infty$ & the normalized non-Archimedean norm with $|\theta|_\infty:=q$. \\
    $K_\infty$ & $\laurent{\FF_q}{1/\theta}$, the completion of $K$ with respect to $|\cdot|_\infty$. \\
    $\CC_{\infty}$ & the completion of an algebraic closure of $K_\infty$. \\
    $\oK$ & the algebraic closure of $K$ inside $\CC_{\infty}$.  \\
    $\KK$ & any algebraically closed field with $K\subset\KK\subset\mathbb{C}_\infty$.\\
    $\TT_\alpha$ & $\{\sum_{i\geq 0}a_it^i\in\mathbb{C}_\infty\llbracket t\rrbracket\mid\lim_{i\to\infty}|a_i\alpha^i|_\infty=0\}$, the Tate algebra of the closed disk in $\CC_{\infty}$ with radius $|\alpha|_\infty$. We denote by $\TT:=\TT_1$\\
    $\EE$ &  the subring of $\TT$ consisting of power series $\sum_{i\geq 0}a_it^i\in\oK\llbracket t\rrbracket$ converges everywhere on $\mathbb{C}_\infty$ and $[K_\infty(a_0,a_1,\dots):K_\infty]<\infty$.\\
    $\|\cdot \|$ & the Gauss norm on $\TT$ defined by $\|\sum_{i\geq 0}a_it^i\|:=\sup_{i\geq 0}\{|a_i|_\infty\}$ with $\sum_{i\geq 0}a_it^i\in\TT$.
    \end{longtable}

    For $n\in\mathbb{Z}$, we define the $n$-fold Frobenius twisting on the Laurent series field $\mathbb{C}_\infty(\!(t)\!)$ by setting
    \begin{align}\label{Eq:Twisting}
        \begin{split}
            \mathbb{C}_\infty(\!(t)\!)&\to\mathbb{C}_\infty(\!(t)\!)\\
        f=\sum c_it^i&\mapsto f^{(n)}:=\sum c_i^{q^n}t^i.
        \end{split}
    \end{align}
    We set $\KK[t,\sigma]:=\KK[t][\sigma]$ to be the twisted polynomial ring in variable $\sigma$ with coefficients in $\KK[t]$ subject to the relation $\sigma f=f^{(-1)}\sigma$. Finally, we extend $|\cdot|_\infty$ to matrices as follows. For $B=(B_{ij})\in\Mat_{m\times n}(\mathbb{C}_\infty)$, we define $|B|_\infty:=\max\{|B_{ij}|_\infty\}$.

\subsection{Hyperderivatives and expansions of polynomials}
    Let $f\in\KK[t]$ and $j\geq 0$ be an integer. Consider the $\KK$-linear operator ${\rm{d}}_t^j:\KK[t]\to\KK[t]$ defined by ${\rm{d}}_t^j(t^m):=\binom{m}{j}t^{m-j}$. It is called the $j$-th hyperderivative on $\KK[t]$. Given $\alpha\in\KK$, the hyperderivatives are related to the expansion of $f$ at $\alpha$. More precisely, we have
    \[
        f=\big({\rm{d}}_t^0f\big)\mid_{t=\alpha}+\big({\rm{d}}_t^1f\big)\mid_{t=\alpha}(t-\alpha)+\cdots+\big({\rm{d}}_t^nf\big)\mid_{t=\alpha}(t-\alpha)^n.
    \]
    In the present paper, we are interested in the case of $\alpha=0$ and $\alpha=\theta$. The following lemma is straightforward.
    \begin{lemma}\label{Lem:SwitchExpansion}
        Let $f\in\KK[t]$ with $\deg_t(f)=n\geq 0$. Then we have
        \[
            \begin{pmatrix}
                {\rm{d}}_t^nf\\
                {\rm{d}}_t^{n-1}f\\
                \vdots\\
                {\rm{d}}_t^1f\\
                {\rm{d}}_t^0f
            \end{pmatrix}\mid_{t=0}=\begin{pmatrix}
                1 & 0 & \cdots & 0 & 0\\
                -\theta & 1 &  &  & 0\\
                \vdots & \vdots & \ddots & & \vdots\\
                (-1)^{n-1}\theta^{n-1} & (-1)^{n-2}\binom{n-1}{1}\theta^{n-1} & & 1 & 0\\
                (-1)^n\theta^n & (-1)^{n-1}\binom{n}{1}\theta^{n} & \cdots & -\binom{n}{n-1}\theta & 1
            \end{pmatrix}\begin{pmatrix}
                {\rm{d}}_t^nf\\
                {\rm{d}}_t^{n-1}f\\
                \vdots\\
                {\rm{d}}_t^1f\\
                {\rm{d}}_t^0f
            \end{pmatrix}\mid_{t=\theta}.
        \]
    \end{lemma}

    \begin{proof}
        On the one hand,
        \[
            f=\big({\rm{d}}_t^0f\big)\mid_{t=0}+\big({\rm{d}}_t^1f\big)\mid_{t=0}t+\cdots+\big({\rm{d}}_t^nf\big)\mid_{t=0}t^n.
        \]
        On the other hand,
        \[
            f=\big({\rm{d}}_t^0f\big)\mid_{t=\theta}+\big({\rm{d}}_t^1f\big)\mid_{t=\theta}(t-\theta)+\cdots+\big({\rm{d}}_t^nf\big)\mid_{t=\theta}(t-\theta)^n.
        \]
        By using the binomial theorem, we deduce that
        \begin{align*}
            f&=\sum_{i=0}^{n}({\rm{d}}_t^if)\mid_{t=\theta}(t-\theta)^i\\
            &=\sum_{i=0}^{n}({\rm{d}}_t^if)\mid_{t=\theta}\sum_{j=0}^i\binom{i}{j}(-1)^{i-j}\theta^{i-j}t^{j}\\
            &=\sum_{j=0}^{n-1}\bigg(\sum_{i=j}^{n-1}\binom{i}{j}(-1)^{i-j}\theta^{i-j}({\rm{d}}_t^if)\mid_{t=\theta}\bigg)t^{j}.
        \end{align*}
        In particular, for each $0\leq j\leq n$ we obtain
        \[
            ({\rm{d}}_t^jf)\mid_{t=0}=\sum_{i=j}^{n-1}\binom{i}{j}(-1)^{i-j}\theta^{i-j}({\rm{d}}_t^if)\mid_{t=\theta}.
        \]
        The matrix form of the above relations provide the desired result.
    \end{proof}

    For the convenience of later use, for each $n\geq 0$, we define
    \[
        \mathcal{D}_n:=\begin{pmatrix}
                1 & 0 & \cdots & 0 & 0\\
                -\theta & 1 &  &  & 0\\
                \vdots & \vdots & \ddots & & \vdots\\
                (-1)^{n-1}\theta^{n-1} & (-1)^{n-2}\binom{n-1}{1}\theta^{n-1} & & 1 & 0\\
                (-1)^n\theta^n & (-1)^{n-1}\binom{n}{1}\theta^{n} & \cdots & -\binom{n}{n-1}\theta & 1
            \end{pmatrix}.
    \]
    Note that $\KK[t]\subset\TT$ and for $f\in\KK[t]$ we have
    \[
        \|f\|=\max_{i=0}^n\{|{(\rm{d}}_t^if)\mid_{t=0}|_\infty\}.
    \]
    We further set
    \[
        \|f\|_\theta:=\max_{i=0}^n\{|{(\rm{d}}_t^if)\mid_{t=\theta}|_\infty\}.
    \]
    As a consequence of Lemma~\ref{Lem:SwitchExpansion}, we have the following result on different norms of $\KK[t]$.
    \begin{lemma}\label{Lem:CompareNorms}
        Let $\{f_i\}_{i=0}^\infty\subset\KK[t]$ be a sequence of polynomials with $\deg_t(f_i)=n\geq 0$. If $\lim_{i\to\infty}\|f_i\|_\theta=0$, then $\lim_{i\to\infty}\|f_i\|=0$.
    \end{lemma}

    \begin{proof}
        By Lemma~\ref{Lem:SwitchExpansion}, we have
        \[
            \|f_i\|\leq|\mathcal{D}_n|_\infty\|f_i\|_\theta.
        \]
        The desired result follows immediately by taking the limit to infinity on the both sides of the above inequality.
    \end{proof}

\subsection{Frobenius modules and rigid analytic trivializations}
    In what follows, we follow \cite{CPY19} closely to recall the essential background about the Frobenius modules.
    \begin{definition}
        A Frobenius module is a $\KK[t,\sigma]$-module that is free of finite rank over $\KK[t]$. 
    \end{definition}
    Given a Frobenius module $M$ with a $\KK[t]$-basis $\{\mathbf{m}_1,\dots,\mathbf{m}_r\}\subset M$, there exists a matrix $\Phi_M\in\Mat_r(\KK[t])$ such that $\sigma\cdot(\mathbf{m}_1,\dots,\mathbf{m}_r)^\tr=\Phi_M(\mathbf{m}_1,\dots,\mathbf{m}_r)^\tr$. We call $M$ the Frobenius module defined by $\Phi_M$. Conversely, given any $\Phi\in\Mat_r(\KK[t])$, we can construct a Frobenius module defined by $\Phi$ as follows. Let $M_\Phi:=\Mat_{r\times 1}(\KK[t])$. Then we define the $\KK[\sigma]$-module structure on $M_\Phi$ by setting $\sigma\cdot(a_1,\dots,a_r):=(a_1^{(-1)},\dots,a_r^{(-1)})\Phi$ for $(a_1,\dots,a_r)\in M_\Phi$. One checks directly that $M_\Phi$ is indeed the Frobenius module defined by $\Phi$.
    
    A typical example is the $n$-th tensor powers of the Carlitz dual $t$-motive which plays a crucial role in this paper.
    \begin{example}
        For an integer $n\geq 1$, let $M_{C^{\otimes n}}=\KK[t]$ be the free $\KK[t]$-module of rank $1$ equipped with the $\KK[\sigma]$-module structure given by $\sigma\cdot f:=f^{(-1)}(t-\theta)^n$. Then $M_{C^{\otimes n}}$ is a Frobenius module defined by $(t-\theta)^n$. Moreover, $M_{C^{\otimes n}}$ is a free $\KK[\sigma]$-module of rank $n$ with a canonical choice of the $\KK[\sigma]$-basis $\{1,(t-\theta),\dots,(t-\theta)^{n-1}\}$.
    \end{example}

    A companion notion related to Frobenius modules is rigid analytic triviality.
    \begin{definition}
        We say a Frobenius module $M$ defined by $\Phi\in\Mat_r(\KK[t])$ rigid analytically trivial if there exists $\Psi\in\GL_r(\TT)$ such that
        \[
            \Psi^{(-1)}=\Phi\Psi.
        \]
        The matrix $\Psi$ is called a rigid analytic trivialization of $\Phi$.
    \end{definition}

    Let $M$ be a rigid analytically trivial Frobenius module defined by $\Phi\in\Mat_r(\KK[t])$. We set
    \[
        \mathrm{Sol}(\Phi):=\{\Psi\in\GL_r(\TT)\mid\Psi^{(-1)}=\Phi\Psi\}
    \]
    to be the collection of all rigid analytic trivializations of $\Phi$. Since
    \begin{equation}\label{Eq:SigmaInvariant}
        \GL_r(\TT)^\sigma:=\{B\in\GL_r(\TT)\mid B^{(-1)}=B\}=\GL_r(\mathbb{F}_q[t]),
    \end{equation}
    The set $\mathrm{Sol}(\Phi)$ naturally has a right action of $\GL_r(\mathbb{F}_q[t])$. Indeed, given any $\Psi\in\mathrm{Sol}(\Phi)$ and $B\in\GL_r(\mathbb{F}_q[t])$, we have
    \[
        (\Psi B)^{(-1)}=\Phi(\Psi B).
    \]
    Given $\Psi_1,\Psi_2\in\GL_r(\TT)$, one checks directly that
    \[
        (\Psi_1^{-1}\Psi_2)^{(-1)}=\Psi_1^{-1}\Phi^{-1}\Phi\Psi_2=\Psi_1^{-1}\Psi_2.
    \]
    It follows by \eqref{Eq:SigmaInvariant} that $\Psi_1^{-1}\Psi_2\in\GL_r(\mathbb{F}_q[t])$. Hence, the right $\GL_r(\mathbb{F}_q[t])$-action on $\mathrm{Sol}(\Phi)$ is transitive and we have
    \[
        \mathrm{Sol}(\Phi)=\Psi\GL_r(\mathbb{F}_q[t])
    \]
    for any $\Psi\in\mathrm{Sol}(\Phi)$.

    \begin{example}
        For an integer $n\geq 1$, the Frobenius module $M_{C^{\otimes n}}$ defined by $(t-\theta)^n$ is rigid analytic trivial. A rigid analytically trivialization of $(t-\theta)^n$ can be described explicitly by the $n$-th power of the following infinite product (cf. \cite[Prop.~3.3.6]{Pap08})
        \[
            \Omega:=(-\theta)^{\frac{-q}{q-1}}\prod_{i=1}^\infty\left(1-\frac{t}{\theta^{q^i}}\right)\in\TT^\times.
        \]
        In particular, we have
        \[
            \mathrm{Sol}\big((t-\theta)^n\big)=\Omega^n\cdot\mathbb{F}_q^\times.
        \]
    \end{example}

    We finish this subsection by stating an essential property of $\mathrm{Sol}(\Phi)$ that is an immediate consequence of \cite[Prop.~3.1.3]{ABP04}.

    \begin{proposition}\label{Prop:ABP_P313}
        Let $M$ be a rigid analytically trivial Frobenius module defined by the matrix $\Phi\in\Mat_r(\KK[t])$ with $(\det\Phi)\mid_{t=0}\neq 0$. Then
        \[
            \mathrm{Sol}(\Phi)\subset\Mat_r(\EE).
        \]
    \end{proposition}

\section{Furusho's analytic continuation}
    The main theme in this section is to establish an analytic continuation of Kochubei polylogarithms using the techniques initiated by Furusho in \cite{Fur22}.

\subsection{Analytic continuation of Kochubei polylogarithms}
    For $n\in\mathbb{Z}_{\geq 0}$, recall that the $n$-th Kochubei polylogarithm $\Li_{\mathcal{K},n}(z)$ has been defined in \eqref{Eq:KPLs_Def}.
    It defines an analytic function on $\mathbb{C}_\infty$ convergent at $z_0\in\mathbb{C}_\infty$ with $|z_0|_\infty<q^n$ and satisfies \eqref{Eq:DifferenceEquation_n>1}.
    
    In what follows, inspired by the $t$-deformation series of $\Li_{\mathcal{K},n}(z)$ introduced by Harada in \cite{Har22}, we are able to apply Furusho's analytic continuation method to our setting. Let $n\geq 0$ be a non-negative integer and $f\in\KK[t]$ with $\|f\|<q^n$. We have
    \[
        \mathscr{L}_{f,n}:=\sum_{i\geq 1}\frac{f^{(i)}}{(\theta^{q^i}-t)^n}\in\TT.
    \]
    This is a slightly generalized version of Harada's $t$-deformation series given in \cite[Def.~2.1]{Har22}. One checks directly that
    \[
        \mathscr{L}_{f,n}^{(-1)}=\frac{(-1)^nf}{(t-\theta)^n}+\mathscr{L}_{f,n}.
    \]
    If $f=u\in\KK$ with $|u|_\infty<q^n$, then it is regular at $t=\theta$, and its specialization recovers the $n$-th Kochubei polylogarithm at $u$
    \begin{equation}\label{Eq:Ev_KPLs}
        \mathscr{L}_{f,n}\mid_{t=\theta}=\Li_{\mathcal{K},n}(u)\in\mathbb{C}_\infty.
    \end{equation}
    
    Consider a variant of Furusho's $\wp$-function
    \begin{align*}
        \wp:\TT&\to\TT\\
        f&\mapsto f^{(-1)}-f.
    \end{align*}
    It is surjective with $\Ker\wp=\mathbb{F}_q[t]$. Thus, we have the induced map $\wp^{-1}:\TT\to\TT/\mathbb{F}_q[t]$. Since $t-\theta$ is a unit in $\TT$, we have $1/(t-\theta)^n\in\TT$. It follows that given $f\in\KK[t]$, we can define
    \[
        \overset{\rightarrow}{\mathscr{L}_{f,n}}:=\wp^{-1}\big(\frac{(-1)^nf}{(t-\theta)^n}\big)\in\TT/\mathbb{F}_q[t].
    \]
    Note that if $f=u\in\KK$ with $|u|_\infty<q^n$, we have
    \begin{equation}\label{Eq:Small_Ev_KPLs}
        \overset{\rightarrow}{\mathscr{L}_{f,n}}=\mathscr{L}_{f,n}+\mathbb{F}_q[t]\in\TT_\theta/\mathbb{F}_q[t]
    \end{equation}
    which is regular at $t=\theta$. For general $f\in\KK[t]$, it is not immediate to see whether $\overset{\rightarrow}{\mathscr{L}_{f,n}}$ is regular at $t=\theta$ or not, since the rational function $(-1)^nf/(t-\theta)^n$ may have a pole of order at most $n$ at $t=\theta$. We can resolve this issue by relating elements in $\overset{\rightarrow}{\mathscr{L}_{f,n}}$ to solutions of a specific Frobenius difference equation.

    \begin{proposition}\label{Prop:Conv}
        Let $f\in\KK[t]$. Then
        \[
            \overset{\rightarrow}{\mathscr{L}_{f,n}}=\wp^{-1}\big(\frac{(-1)^nf}{(t-\theta)^n}\big)\in\TT_\theta/\mathbb{F}_q[t].
        \]
    \end{proposition}

    \begin{proof}
        Let $\mathscr{L}_{f,n}^\circ\in\wp^{-1}\big(\frac{(-1)^nf}{(t-\theta)^n}\big)$. It is enough to show that $\mathscr{L}_{f,n}^\circ\in\TT_\theta$. One checks directly that
        \[
            \big(\Omega^n\mathscr{L}_{f,n}^\circ\big)^{(-1)}=(-1)^nf+(t-\theta)^n\big(\Omega^n\mathscr{L}_{f,n}^\circ\big).
        \]
        In particular, we have
        \[
            \begin{pmatrix}
                \Omega^n & \\
             \Omega^n\mathscr{L}_{f,n}^\circ & \Omega^n
            \end{pmatrix}^{(-1)}=\begin{pmatrix}
                (t-\theta)^n & \\
                (-1)^nf & (t-\theta)^n
            \end{pmatrix}\begin{pmatrix}
                \Omega^n & \\
                \Omega^n\mathscr{L}_{f,n}^\circ & \Omega^n
            \end{pmatrix}.
        \]
        Thus, if we set
        \[
            \Phi_{(-1)^nf}:=\begin{pmatrix}
                (t-\theta)^n & \\
                (-1)^nf & (t-\theta)^n
            \end{pmatrix}
        \]
        then we have
        \begin{align*}
            \begin{pmatrix}
            \Omega^n & \\
            \Omega^n\mathscr{L}_{f,n}^\circ & \Omega^n
            \end{pmatrix}\in\mathrm{Sol}(\Phi_{(-1)^nf}).
        \end{align*}
        Since $(\det\Phi_{(-1)^nf})\mid_{t=0}\neq 0$, it follows from Prop.~\ref{Prop:ABP_P313} that $\Omega^n\mathscr{L}_{f,n}^\circ\in\EE\subset\TT_\theta$. By using $\Omega\in\TT_\theta^\times$, we conclude
        \[
            \mathscr{L}_{f,n}^\circ\in\TT_\theta.
        \]
        The desired result now follows.
    \end{proof}

    
    It follows from Proposition~\ref{Prop:Conv} that we have the following well-defined evaluation.
    \begin{definition}
        For $f\in\KK[t]$, we define
        \[
            \overset{\rightarrow}{\Li_{\mathcal{K},n}}(f):=\overset{\rightarrow}{\mathscr{L}_{f,n}}\mid_{t=\theta}\in\mathbb{C}_\infty/A.
        \]
        In particular, since $\KK$ can be any arbitrary algebraically closed field with $K\subset\KK\subset\mathbb{C}_\infty$, we can specialize in $\KK=\mathbb{C}_\infty$ and use the natural inclusion $\mathbb{C}_\infty\subset\mathbb{C}_\infty[t]$ to define
        \begin{align*}
            \overset{\rightarrow}{\Li_{\mathcal{K},n}}(\cdot):\mathbb{C}_\infty&\to\mathbb{C}_\infty/A\\
            u&\mapsto\overset{\rightarrow}{\mathscr{L}_{u,n}}\mid_{t=\theta}
        \end{align*}
        Following \cite{Fur22}, we refer any $\mathbb{F}_q$-linear lift $\Li_{\mathcal{K},n}^\circ:\mathbb{C}_\infty\to\mathbb{C}_\infty$ as a branch of $\overset{\rightarrow}{\Li_{\mathcal{K},n}}$.
    \end{definition}
    It is natural to ask if $\overset{\rightarrow}{\Li_{\mathcal{K},n}}(\cdot):\mathbb{C}_\infty\to\mathbb{C}_\infty/A$ satisfies the same difference equations \eqref{Eq:DifferenceEquation_n>1}.
    To verify it, we construct a $t$-motivic counterpart of the Carlitz difference operator $\Delta$. For $f\in\KK[t]$ and $n\in\mathbb{Z}$ with $n\geq 1$, we define
    \[
        \bm{\delta}(\overset{\rightarrow}{\mathscr{L}_{f,n}}):=\overset{\rightarrow}{\mathscr{L}_{\theta f,n}}-t\overset{\rightarrow}{\mathscr{L}_{f,n}}.
    \]
    It is a straightforward calculation that
    \[
        \wp\bigg(\bm{\delta}(\overset{\rightarrow}{\mathscr{L}_{f,n}})\bigg)=\wp\bigg(\overset{\rightarrow}{\mathscr{L}_{\theta f,n}}-t\overset{\rightarrow}{\mathscr{L}_{f,n}}\bigg)=\frac{(-1)^n\theta f}{(t-\theta)^n}-\frac{(-1)^ntf}{(t-\theta)^n}=\frac{(-1)^{n-1}f}{(t-\theta)^{n-1}}.
    \]
    Thus, we obtain
    \begin{equation}\label{Eq:t_Motivic_CDO}
        \bm{\delta}(\overset{\rightarrow}{\mathscr{L}_{f,n}})=\overset{\rightarrow}{\mathscr{L}_{f,n-1}}
    \end{equation}
    
    Now we are ready to prove the main results for our analytically continued Kochubei polylogarithms.
    \begin{theorem}\label{Thm:KPLs}
        Let $u\in\mathbb{C}_\infty$ and $n\geq 0$ be an integer. Then the following assertions hold.
        \begin{enumerate}
            \item If $|u|_\infty<q^n$, then $\overset{\rightarrow}{\Li_{\mathcal{K},n}}(u)\equiv\Li_{\mathcal{K},n}(u)~(\mathrm{mod}~A)$.
            \item $\overset{\rightarrow}{\Li_{\mathcal{K},n}}$ locally admits an analytic lift in the sense of Theorem~\ref{Thm:Intro_I}(2).
            \item For $n\geq 1$, we have
            \[
                (\Delta\overset{\rightarrow}{\Li_{\mathcal{K},n}})(u):=\overset{\rightarrow}{\Li_{\mathcal{K},n}}(\theta u)-\theta\overset{\rightarrow}{\Li_{\mathcal{K},n}}(u)=\overset{\rightarrow}{\Li_{\mathcal{K},n-1}}(u).
            \]
        \end{enumerate}
    \end{theorem}

    \begin{proof}
        The first assertion follows by \eqref{Eq:Ev_KPLs} and \eqref{Eq:Small_Ev_KPLs} immediately. For the second statement, note that $\overset{\rightarrow}{\Li_{\mathcal{K},n}}$ is additive and by the first part that
        \[
            \overset{\rightarrow}{\Li_{\mathcal{K},n}}(u)\equiv\Li_{\mathcal{K},n}(u)~(\mathrm{mod}~A)
        \]
        whenever $|u|_\infty<q^n$. The desired property now comes from the fact that $\Li_{\mathcal{K},n}$ is defined by a power series converging on the closed disk of radius smaller than $q^n$ and $A$ is discrete inside $\mathbb{C}_\infty$.

        The third assertion is a consequence of \eqref{Eq:t_Motivic_CDO}. Indeed, for $n\geq 1$, we have
        \begin{align*}
            (\Delta\overset{\rightarrow}{\Li_{\mathcal{K},n}})(u)&=\overset{\rightarrow}{\Li_{\mathcal{K},n}}(\theta u)-\theta\overset{\rightarrow}{\Li_{\mathcal{K},n}}(u)\\
            &=(\overset{\rightarrow}{\mathscr{L}_{\theta u,n}})\mid_{t=\theta}-(t\overset{\rightarrow}{\mathscr{L}_{u,n}})_{t=\theta}\\
            &=\bm{\delta}(\overset{\rightarrow}{\mathscr{L}_{u,n}})\mid_{t=\theta}\\
            &=\overset{\rightarrow}{\mathscr{L}_{u,n-1}}\mid_{t=\theta}\\
            &=\overset{\rightarrow}{\Li_{\mathcal{K},n-1}}(u)
        \end{align*}
        which gives the desired result.
    \end{proof}

\subsection{Kochubei multiple polylogarithms}
    In what follows, we aim to carry out an analytic continuation of Kochubei multiple polylogarithms, abbreviated as KMPLs, using the similar techniques developed in the previous subsection. 
    For $\fs=(s_1,\dots,s_r)\in\mathbb{Z}_{\geq 0}^r$, recall that the $\fs$-th KMPL $\Li_{\mathcal{K},\fs}(z_1,\dots,z_r)$ has been defined in \eqref{Eq:KMPLs_Def}.
    It defines an analytic function on $\Mat_{1\times r}(\mathbb{C}_\infty)$ converging at $\bu\in\mathbb{D}_\fs$ where $\mathbb{D}_\fs$ is defined in \eqref{Eq:D_fs}. In particular, it converges at $(u_1,\dots,u_r)\in\Mat_{1\times r}(\mathbb{C}_\infty)$ with $|u_i|_\infty<q^{s_i}$ for each $1\leq i\leq r$. 
    
    Similar to the case of KPLs, we slightly generalize Harada's $t$-deformation series given in \cite[Def.~3.9]{Har22} as follows. Let $\fs=(s_1,\dots,s_r)\in\mathbb{Z}_{\geq 0}^r$ and $\ff=(f_1,\dots,f_r)\in\Mat_{1\times r}(\KK[t])$ with
    \[
        \|f_1/\theta^{s_1}\|^{q^{i_1}}_\infty\cdots\|f_r/\theta^{s_r}\|^{q^{i_r}}_\infty\to 0,~\mathrm{as}~0<i_r<\cdots<i_1\to\infty.
    \] 
    Consider
    \[
        \mathscr{L}_{\ff,\fs}:=\sum_{i_1>\cdots>i_r>0}\frac{f_1^{(i_1)}\cdots f_r^{(i_r)}}{(\theta^{q^{i_1}}-t)^{s_1}\cdots(\theta^{q^{i_r}}-t)^{s_r}}\in\TT.
    \]
    It satisfies the difference equation
    \[
        \mathscr{L}_{\ff,\fs}^{(-1)}=\frac{(-1)^{s_r}f_r}{(t-\theta)^{s_r}}\mathscr{L}_{\ff[1,r-1],\fs[1,r-1]}+\mathscr{L}_{\ff,\fs},
    \]
    where for each $0\leq i,j\leq r$ we define
    \[
        (\ff[i,j],\fs[i,j]):=\begin{cases}
            \big((f_i,\dots,f_j),(s_i,\dots,s_j)\big),~&\mathrm{if}~i\leq j\\
            (\emptyset,\emptyset),~&\mathrm{if}~i>j
        \end{cases}
    \]
    and $\mathscr{L}_{\emptyset,\emptyset}:=1$.
    If $\ff=\bu\in\mathbb{D}_\fs$, then it is regular at $t=\theta$ and its specialization recovers the $\fs$-th Kochubei multiple polylogarithm at $\bu$
    \begin{equation}\label{Eq:Ev_KMPLs}
        \mathscr{L}_{\ff,\fs}\mid_{t=\theta}=\Li_{\mathcal{K},\fs}(\bu)\in\mathbb{C}_\infty.
    \end{equation}

    Note that for any $\ff=(f_1,\dots,f_r)\in\Mat_{1\times r}(\KK[t])$, the system of difference equations
    \begin{equation}\label{Eq:KMPLs_System}
        \begin{cases}
            \wp(F_1)=F_1^{(-1)}-F_1=\frac{(-1)^{s_1}f_1}{(t-\theta)^{s_1}}\\
            \wp(F_2)=F_2^{(-1)}-F_2=\frac{(-1)^{s_2}f_2}{(t-\theta)^{s_2}}F_1\\
            \vdots\\
            \wp(F_r)=F_r^{(-1)}-F_r=\frac{(-1)^{s_r}f_r}{(t-\theta)^{s_r}}F_{r-1}
        \end{cases}
    \end{equation}
    is always solvable in $\Mat_{r\times 1}(\TT)$. In particular, if $\ff=\bu\in\mathbb{D}_\fs$, then the above system admits the following solution
    \[
        \begin{pmatrix}
            F_1\\
            \vdots\\
            F_r
        \end{pmatrix}=\begin{pmatrix}
            \mathscr{L}_{\ff[1,1],\fs[1,1]}\\
            \vdots\\
            \mathscr{L}_{\ff[1,r-1],\fs[1,r-1]}\\
            \mathscr{L}_{\ff,\fs}
        \end{pmatrix}\in\Mat_{r\times 1}(\TT_\theta).
    \]
    Our next task is to determine the difference between two solutions of \eqref{Eq:KMPLs_System} and their regularity at $t=\theta$ in general case. For convenience in later use, we will call \eqref{Eq:KMPLs_System} the system of difference equations associated with the pair $(\ff,\fs)$. Moreover, we set
    \[
        \mathrm{Sol}(\ff,\fs):=\{(F_1,\dots,F_r)^\tr\in\Mat_{r\times 1}(\TT)\mid(F_1,\dots,F_r)^\tr~\mathrm{satisfies}~\eqref{Eq:KMPLs_System}\}\subset\Mat_{r\times 1}(\TT).
    \]
    For each $2\leq j\leq r$, we fix a choice of $(\psi_{jj},\dots,\psi_{jr})^\tr\in\mathrm{Sol}(\ff[j,r],\fs[j,r])\subset\Mat_{(r-j+1)\times 1}(\TT)$. Consider
    \[
        \bm{\psi}_2:=\begin{pmatrix}
            1\\
            \psi_{22}\\
            \psi_{23}\\
            \vdots\\
            \psi_{2r}
        \end{pmatrix},~\bm{\psi}_3:=\begin{pmatrix}
            0\\
            1\\
            \psi_{33}\\
            \vdots\\
            \psi_{3r}
        \end{pmatrix},\dots,~\bm{\psi}_r:=\begin{pmatrix}
            0\\
            \vdots\\
            0\\
            1\\
            \psi_{rr}
        \end{pmatrix},~\bm{\psi}_{r+1}:=\begin{pmatrix}
            0\\
            \vdots\\
            0\\
            1
        \end{pmatrix}\in\Mat_{r\times 1}(\TT)
    \]
    and
    \[
        \mathbb{M}_{\widetilde{\ff},\widetilde{\fs}}:=\mathrm{Span}_{\mathbb{F}_q[t]}\{\bm{\psi}_2,\dots,\bm{\psi}_{r+1}\}\subset\Mat_{r\times 1}(\TT),
    \]
    where $\widetilde{\ff}:=\ff[2,r]$ and $\widetilde{\fs}:=\fs[2,r]$.
    \begin{proposition}
        Let $\fs=(s_1,\dots,s_r)\in\Mat_{1\times r}(\mathbb{Z}_{\geq 0})$ and $\ff=(f_1,\dots,f_r)\in\Mat_{1\times r}(\KK[t])$. Then the following assertions hold.
        \begin{enumerate}
            \item The $\mathbb{F}_q[t]$-module $\mathbb{M}_{\widetilde{\ff},\widetilde{\fs}}$ is independent of the choice of $\{(\psi_{jj},\dots,\psi_{jr})^\tr\}_{j=2}^{r}$.
            \item We have $\mathrm{Sol}(\ff,\fs)\subset\Mat_{r\times 1}(\TT_\theta)$. Consequently, $\mathbb{M}_{\widetilde{\ff},\widetilde{\fs}}\subset\Mat_{r\times 1}(\TT_\theta)$.
            \item Let $(F_1,\dots,F_r)^\tr,(G_1,\dots,G_r)^\tr\in\mathrm{Sol}(\ff,\fs)$. Then
            \[
                (F_1,\dots,F_r)^\tr-(G_1,\dots,G_r)^\tr\in\mathbb{M}_{\widetilde{\ff},\widetilde{\fs}}.
            \]
        \end{enumerate}
    \end{proposition}

    \begin{proof}
        Let $m(\fs):=\max_{1\leq i\leq r}\{s_i\}$ and
        \[
            \Phi':=(t-\theta)^{m(\fs)}
            \begin{pmatrix}
                1 & 0 & \cdots & 0 \\
                \frac{(-1)f_2}{(t-\theta)^{s_2}} & \ddots & & \vdots\\
                 & \ddots & 1 & 0\\
                 & & \frac{(-1)f_r}{(t-\theta)^{s_r}} & 1
            \end{pmatrix}\in\Mat_{r}(\KK[t]).
        \]
        Then
        \[
            \Psi':=\Omega^{m(\fs)}\bigg(\bm{\psi}_2,\dots,\bm{\psi}_{r+1}\bigg)
        \]
        satisfies the difference equation $\Psi'^{(-1)}=\Phi'\Psi'$, and thus $\Psi'\in\mathrm{Sol}(\Phi')$.
        For each $2\leq j\leq r$, if we have another choice of $(\widetilde{\psi}_{jj},\dots,\tilde{\psi}_{jr})^\tr\in\mathrm{Sol}(\ff[j,r],\fs[j,r])$, then we can similarly consider
        \[
            \bm{\tilde{\psi}}_2:=\begin{pmatrix}
            1\\
            \tilde{\psi}_{22}\\
            \tilde{\psi}_{23}\\
            \vdots\\
            \tilde{\psi}_{2r}
        \end{pmatrix},~\bm{\tilde{\psi}}_3:=\begin{pmatrix}
            0\\
            1\\
            \tilde{\psi}_{33}\\
            \vdots\\
            \tilde{\psi}_{3r}
        \end{pmatrix},\dots,~\bm{\tilde{\psi}}_r:=\begin{pmatrix}
            0\\
            \vdots\\
            0\\
            1\\
            \tilde{\psi}_{rr}
        \end{pmatrix},~\bm{\tilde{\psi}}_{r+1}:=\begin{pmatrix}
            0\\
            \vdots\\
            0\\
            1
        \end{pmatrix}\in\Mat_{r\times 1}(\TT).
        \]
        It follows that
        \[
            \tilde{\Psi}':=\Omega^{m(\fs)}\bigg(\bm{\tilde{\psi}}_2,\dots,\bm{\tilde{\psi}}_{r+1}\bigg)
        \]
        satisfies the same difference equation $\tilde{\Psi}'^{(-1)}=\Phi'\tilde{\Psi}'$, and thus $\tilde{\Psi}'\in\mathrm{Sol}(\Phi')$. Since the right $\GL_r(\mathbb{F}_q[t])$-action on $\mathrm{Sol}(\Phi')$ is transitive, there is $C'\in\GL_r(\mathbb{F}_q[t])$ so that $\tilde{\Psi}'=\Psi'C'$. Hence $\mathbb{M}_{\widetilde{\ff},\widetilde{\fs}}$ is independent of the choice of $\{(\psi_{jj},\dots,\psi_{jr})^\tr\}_{j=2}^{r}$.
    
        To prove the second assertion on the regularity at $t=\theta$, let
        \[
            \Phi:=\begin{pmatrix}
                (t-\theta)^{m(\fs)} & 0 \\
                (-1)^{s_1}f_1{(t-\theta)^{m(\fs)-s_1}} & \Phi'
            \end{pmatrix}\in\Mat_{r+1}(\KK[t]).
        \]
        For each $(F_1,\dots,F_r)^\tr\in\mathrm{Sol}(\ff,\fs)$, if we set
        \[
            \bm{\psi}_1:=(F_1,\dots,F_r)^\tr,
        \]
        then
        \[
            \Psi:=\begin{pmatrix}
                \Omega^{m(\fs)} & 0\\
                \Omega^{m(\fs)}\bm{\psi}_1 & \Psi'
            \end{pmatrix}\in\GL_{r+1}(\TT)
        \]
        satisfies $\Psi^{(-1)}=\Phi\Psi$. In particular, $\Psi\in\mathrm{Sol}(\Phi)$. Since $(\det\Phi)\mid_{t=0}\neq 0$, it follows from Prop.~\ref{Prop:ABP_P313} that $\Psi\in\Mat_{r+1}(\EE)$. Since $\Omega\in\TT_\theta^\times$, we conclude that $(F_1,\dots,F_r)^\tr\in\Mat_{r\times 1}(\TT_\theta)$. The second assertion follows immediately.

        Finally, consider $(G_1,\dots,G_r)^\tr\in\mathrm{Sol}(\ff,\fs)$ and set similarly that
        \[
            \bm{\tilde{\psi}}_1:=(G_1,\dots,G_r)^\tr.
        \]
        Then
        \[
            \tilde{\Psi}:=\begin{pmatrix}
                \Omega^{m(\fs)} & 0\\
                \Omega^{m(\fs)}\bm{\tilde{\psi}}_1 & \Psi'
            \end{pmatrix}\in\GL_{r+1}(\TT)
        \]
        satisfies the same difference equation $\tilde{\Psi}^{(-1)}=\Phi\tilde{\Psi}$, and thus $\tilde{\Psi}\in\mathrm{Sol}(\Phi)$. Since the right $\GL_{r+1}(\mathbb{F}_q[t])$-action on $\mathrm{Sol}(\Phi)$ is transitive, there is $C\in\GL_{r+1}(\mathbb{F}_q[t])$ so that $\tilde{\Psi}=\Psi C$. Since $\Psi$ and $\tilde{\Psi}$ are both lower triangular with the same diagonal entries $\Omega^{m(\fs)}$, the matrix $C$ must be of the form
        \[
            C=\begin{pmatrix}
                1 & & & \\
                \star & 1 & & \\
                \vdots & & \ddots & \\
                \star & \cdots & \star & 1
            \end{pmatrix}\in\GL_{r+1}(\mathbb{F}_q[t]).
        \]
        Hence, the desired result follows immediately by comparing the first column of the equation $\tilde{\Psi}=\Psi C$.
    \end{proof}

    It follows that for $\fs\in\Mat_{1\times r}(\mathbb{Z}_{\geq 0})$ and $(f_2,\dots,f_r)\in\Mat_{1\times (r-1)}(\KK[t])$ we can define
    \begin{align*}
        \overset{\rightarrow}{\mathscr{L}_{(-.f_2,\dots,f_r),\fs}}:\KK[t]&\to\Mat_{r\times 1}(\TT_\theta)/\mathbb{M}_{\widetilde{\ff},\widetilde{\fs}}\\
        f_1&\mapsto(F_1,\dots,F_r)^\tr,
    \end{align*}
    where $(F_1,\dots,F_r)^\tr$ is any element in $\mathrm{Sol}(\ff,\fs)$ with $\ff=(f_1,\dots,f_r)$. Note that whenever $\ff=\bu\in\mathbb{D}_\fs$, we have
    \begin{equation}\label{Eq:Small_Ev_KMPLs}
        \overset{\rightarrow}{\mathscr{L}_{\ff,\fs}}=\begin{pmatrix}
            \mathscr{L}_{\ff[1,1],\fs[1,1]}\\
            \vdots\\
            \mathscr{L}_{\ff[1,r-1],\fs[1,r-1]}\\
            \mathscr{L}_{\ff,\fs}
        \end{pmatrix}+\mathbb{M}_{\widetilde{\ff},\widetilde{\fs}}\subset\Mat_{r\times 1}(\TT_\theta).
    \end{equation}
    Due to the regularity at $t=\theta$, we can make the following definition.
    \begin{definition}
        Let $\fs\in\mathbb{Z}_{\geq 0}^r$ and $\widetilde{\bu}=(u_2,\dots,u_r)\in\Mat_{1\times(r-1)}(\mathbb{C}_\infty)$.
        \begin{enumerate}
            \item We define the monodromy module to be the following $A$-submodule inside $\Mat_{r\times 1}(\mathbb{C}_\infty)$
            \[
                \mathcal{M}_{\widetilde{\bu},\widetilde{\fs}}:=\mathbb{M}_{\widetilde{\bu},\widetilde{\fs}}\mid_{t=\theta}.
            \]
            \item We define the $\mathbb{F}_q$-linear map
            \begin{align*}
                \overset{\rightarrow}{\Li_{\mathcal{K},\fs}}(-,u_2,\dots,u_r):\mathbb{C}_\infty&\to\Mat_{r\times 1}(\mathbb{C}_\infty)/\mathcal{M}_{\widetilde{\bu},\widetilde{\fs}}\\
                u_1&\mapsto\overset{\rightarrow}{\mathscr{L}_{\bu,\fs}}\mid_{t=\theta}
            \end{align*}
            where $\bu=(u_1,\dots,u_r)$.
            \item Following \cite{Fur22}, we refer any $\mathbb{F}_q$-linear lift $\Li_{\mathcal{K},\fs}^\circ(-,u_2,\dots,u_r):\mathbb{C}_\infty\to\Mat_{r\times 1}(\mathbb{C}_\infty)$ as a branch of $\overset{\rightarrow}{\Li_{\mathcal{K},\fs}}(-,u_2,\dots,u_r)$.
        \end{enumerate}
    \end{definition}

    Our analytically continued Kochubei multiple polylogarithms satisfy the following properties.

    \begin{theorem}\label{Thm:KMPLs}
        Let $\fs=(s_1,\Tilde{\fs})\in\mathbb{Z}_{\geq 0}^r$ with $\Tilde{\fs}=(s_2,\dots,s_r)\in\mathbb{Z}_{\geq 0}^{r-1}$ and $\Tilde{\bu}=(u_2,\dots,u_r)\in\Mat_{1\times(r-1)}(\mathbb{C}_\infty)$. Then the following assertions hold.
        \begin{enumerate}
            \item For $u_1\in\mathbb{C}_\infty$ with $(u_1,\dots,u_r)\in\mathbb{D}_\fs$, we have
            \[
                \overset{\rightarrow}{\Li_{\mathcal{K},\fs}}(u_1,\dots,u_r)\equiv\begin{pmatrix}
                    \Li_{\mathcal{K},s_1}(u_1)\\
                    \Li_{\mathcal{K},(s_1,s_2)}(u_1,u_2)\\
                    \vdots\\
                    \Li_{\mathcal{K},\fs}(\bu)
                \end{pmatrix}~(\mathrm{mod}~\mathcal{M}_{\widetilde{\bu},\widetilde{\fs}}).
            \]
            \item $\overset{\rightarrow}{\Li_{\mathcal{K},\fs}}(-,u_2,\dots,u_r)$ locally admits an analytic lift as a function on $u_1$ in the sense of Theorem~\ref{Thm:Intro_I}(2).
            \item For $s_1\geq 1$, we have
            \[
                (\Delta\overset{\rightarrow}{\Li_{\mathcal{K},\fs}})(u_1,\dots,u_r)=\overset{\rightarrow}{\Li_{\mathcal{K},(s_1-1,s_2,\dots,s_r)}}(u_1,\dots,u_r).
            \]
        \end{enumerate}
    \end{theorem}

    \begin{proof}
        The first assertion follows immediately by \eqref{Eq:Ev_KMPLs} and \eqref{Eq:Small_Ev_KMPLs}. For the second assertion, note that $\overset{\rightarrow}{\Li_{\mathcal{K},\fs}}(-,u_2,\dots,u_r)$ is additive and whenever $(u_1,\dots,u_r)\in\mathbb{D}_\fs$ we have
        \[
            \overset{\rightarrow}{\Li_{\mathcal{K},\fs}}(u_1,\dots,u_r)\equiv\begin{pmatrix}
                \Li_{\mathcal{K},s_1}(u_1)\\
                \Li_{\mathcal{K},(s_1,s_2)}(u_1,u_2)\\
                \vdots\\
                \Li_{\mathcal{K},\fs}(\bu)
            \end{pmatrix}~(\mathrm{mod}~\mathcal{M}_{\widetilde{\bu},\widetilde{\fs}}),
        \]
        where $\Li_{\mathcal{K},s_1}(u_1),\Li_{\mathcal{K},(s_1,s_2)}(u_1,u_2),\dots,\Li_{\mathcal{K},\fs}(\bu)$ are defined by a power series in $u_1$ on a closed disk. Since $\mathcal{M}_{\widetilde{\bu},\widetilde{\fs}}$ is generated by $\bm{\psi}_2\mid_{t=\theta},\dots,\bm{\psi}_{r+1}\mid_{t=\theta}$ over $A$ and the matrix
        \[
            \bigg(\bm{\psi}_2,\dots,\bm{\psi}_{r+1}\bigg)\mid_{t=\theta}\in\GL_r(\mathbb{C}_\infty)
        \]
        is lower triangular with invertible diagonals, we deduce that $\mathcal{M}_{\widetilde{\bu},\widetilde{\fs}}$ is discrete inside $\Mat_{r\times 1}(\mathbb{C}_\infty)$. The desired property now follows.

        Finally, for $\ff=(f_1,\dots,f_r)\in\Mat_{1\times r}(\KK[t])$ and $\fn=(n_1,\dots,n_r)\in\Mat_{1\times r}(\mathbb{Z}_{\geq 0})$, we consider the operator
        \[
            \bm{\delta}(\overset{\rightarrow}{\mathscr{L}_{\ff,\fn}}):=\overset{\rightarrow}{\mathscr{L}_{(\theta f_1,f_2,\dots,f_r),\fn}}-t\overset{\rightarrow}{\mathscr{L}_{\ff,\fn}}.
        \]
        Assume that $\overset{\rightarrow}{\mathscr{L}_{\ff,\fn}}$ and $\overset{\rightarrow}{\mathscr{L}_{(\theta f_1,f_2,\dots,f_r),\fn}}$ admits a representative $(F_1,\dots,F_r)^\tr$ and $(G_1,\dots,G_r)^\tr$ in $\Mat_{r\times 1}(\TT_\theta)$ respectively. By using the difference equation \eqref{Eq:KMPLs_System} they satisfied, we deduce that
        \[
            \begin{cases}
                \wp(F_1-tG_1)=\frac{(-1)^{n_1}\theta f_1}{(t-\theta)^{n_1}}-\frac{(-1)^{n_1}tf_1}{(t-\theta)^{n_1}}=\frac{(-1)^{n_1-1} f_1}{(t-\theta)^{n_1-1}}\\
                \wp(F_2-tG_2)=\frac{(-1)^{n_2}f_2}{(t-\theta)^{n_2}}F_1-\frac{(-1)^{n_2}f_2}{(t-\theta)^{n_2}}tG_1=\frac{(-1)^{n_2}f_2}{(t-\theta)^{n_2}}\big(F_1-tG_1\big)\\
                \vdots\\
                \wp(F_r-tG_r)=\frac{(-1)^{n_r}f_r}{(t-\theta)^{n_r}}F_{r-1}-\frac{(-1)^{n_r}f_r}{(t-\theta)^{n_r}}tG_{r-1}=\frac{(-1)^{n_r}f_r}{(t-\theta)^{n_r}}\big(F_{r-1}-tG_{r-1}\big)
            \end{cases}.
        \]
        It follows that for $n_1\geq 1$ we have
        \[
            \bm{\delta}(\overset{\rightarrow}{\mathscr{L}_{\ff,\fn}})=\overset{\rightarrow}{\mathscr{L}_{\ff,(n_1-1,n_2,\dots,n_r)}}.
        \]
        Then the desired relation comes from
        \begin{align*}
            (\Delta\overset{\rightarrow}{\Li_{\mathcal{K},\fs}})(\bu)&=\overset{\rightarrow}{\Li_{\mathcal{K},\fs}}(\theta u_1,u_2,\dots,u_r)-\theta\overset{\rightarrow}{\Li_{\mathcal{K},\fs}}(\bu)\\
            &=(\overset{\rightarrow}{\mathscr{L}_{(\theta u_1,u_2,\dots,u_r),\fs}})\mid_{t=\theta}-(t\overset{\rightarrow}{\mathscr{L}_{\bu,\fs}})\mid_{t=\theta}\\
            &=\bm{\delta}(\overset{\rightarrow}{\mathscr{L}_{\bu,\fs}})\mid_{t=\theta}\\
            &=\overset{\rightarrow}{\mathscr{L}_{\bu,(s_1-1,s_2,\dots,s_r)}}\mid_{t=\theta}\\
            &=\overset{\rightarrow}{\Li_{\mathcal{K},(s_1-1,s_2,\dots,s_r)}}(\bu).
        \end{align*}
    \end{proof}

\section{Applications}
    In this section, we will present two applications of our analytically continued KPLs. The first one is the relation with extensions of Frobenius modules, and the second one is the study of linear relations among analytically continued KPLs at algebraic elements.
\subsection{Relations with the structure of $\Ext^1_{\mathscr{F}}(M_{C^{\otimes n}},M_{C^{\otimes n}})$}
    We begin this subsection with a rapid review of extensions of Frobenius modules. We refer readers to \cite[Sec.~2]{CPY19} for more details. Let $\mathscr{F}$ be the category of $\oK[t,\sigma]$-modules that are free of finite rank over $\oK[t]$. Morphisms in $\mathscr{F}$ are left $\oK[t,\sigma]$-module homomorphisms. The category $\mathscr{F}$ is called the category of Frobenius modules. 

    We denote by $\Ext^1_{\mathscr{F}}(M_{C^{\otimes n}},M_{C^{\otimes n}})$ the group of extensions of $M_{C^{\otimes n}}$ by $M_{C^{\otimes n}}$. More precisely, $\Ext^1_{\mathscr{F}}(M_{C^{\otimes n}},M_{C^{\otimes n}})$ consists of equivalent classes of Frobenius modules $[N]$ that fit into the following short exact sequence
    \[
        0\to M_{C^{\otimes n}}\to N\to M_{C^{\otimes n}}\to 0.
    \]
    For $[N_1],[N_2]\in\Ext^1_{\mathscr{F}}(M_{C^{\otimes n}},M_{C^{\otimes n}})$, we have $[N_1]=[N_2]$ if there is a $\oK[t,\sigma]$-module homomorphism $\phi:N_1\to N_2$ so that the following diagram commutes
    \[
        \begin{tikzcd}
            0 \arrow[r] & M_{C^{\otimes n}} \arrow[d, "\Id"] \arrow[r] & N_1 \arrow[d, "\phi"] \arrow[r] & M_{C^{\otimes n}} \arrow[d, "\Id"] \arrow[r] & 0 \\
            0 \arrow[r] & M_{C^{\otimes n}} \arrow[r] & N_2 \arrow[r] & M_{C^{\otimes n}} \ar[r] & 0
        \end{tikzcd}.
    \]
    Our first task is to determine the structure of the $\mathbb{F}_q[t]$-module $\Ext^1_{\mathscr{F}}(M_{C^{\otimes n}},M_{C^{\otimes n}})$. Given $F\in\oK[t]$, we consider
    \[
        \Phi_F:=\begin{pmatrix}
            (t-\theta)^n & \\
            F & (t-\theta)^n
        \end{pmatrix}\in\Mat_2(\oK[t]).
    \]
    Then $\Phi_F$ defines a Frobenius module $N_F=\Mat_{1\times 2}(\oK[t])$ whose $\oK[\sigma]$-module structure is determined by $\sigma\cdot(g_1,g_2):=(g_1,g_2)^{(-1)}\Phi_F$ for any $(g_1,g_2)\in N_F$. One checks that we have the following surjection of $\mathbb{F}_q[t]$-modules
    \begin{align*}
        \delta:\oK[t]&\twoheadrightarrow\Ext^1_{\mathscr{F}}(M_{C^{\otimes n}},M_{C^{\otimes n}})\\
        F&\mapsto [N_F].
    \end{align*}
    Then we can deduce the following identification of $\Ext^1_{\mathscr{F}}(M_{C^{\otimes n}},M_{C^{\otimes n}})$.
    \begin{lemma}\label{Lem:Ext_Identification_1}
        We have an isomorphism of $\mathbb{F}_q[t]$-modules
        \[
            \Ext^1_{\mathscr{F}}(M_{C^{\otimes n}},M_{C^{\otimes n}})\cong\oK[t]/\{B^{(-1)}(t-\theta)^n-(t-\theta)^nB\mid B\in\oK[t]\}.
        \]
    \end{lemma}

    \begin{proof}
        Since $\delta$ is a surjective $\mathbb{F}_q[t]$-module homomorphism, it suffices to determine $\Ker\delta$. Given $F\in\Ker\delta$, note that $\delta(F)=[N_F]$ represents the trivial class in $\Ext^1_{\mathscr{F}}(M_{C^{\otimes n}},M_{C^{\otimes n}})$ if there is a $\oK[t,\sigma]$-module homomorphism $\phi:N_F\to M_{C^{\otimes n}}\oplus M_{C^{\otimes n}}$ so that the following diagram commutes
        \[
            \begin{tikzcd}
                0 \arrow[r] & M_{C^{\otimes n}} \arrow[d, "\Id"] \arrow[r, "\iota_1"] & N_F \arrow[d, "\phi"] \arrow[r, "\pi_1"] & M_{C^{\otimes n}} \arrow[d, "\Id"] \arrow[r] & 0 \\
                0 \arrow[r] & M_{C^{\otimes n}} \arrow[r, "\iota_2"] & M_{C^{\otimes n}}\oplus M_{C^{\otimes n}} \arrow[r, "\pi_2"] & M_{C^{\otimes n}} \ar[r] & 0
            \end{tikzcd}.
        \]
        Let $\mathbf{n}_F$ and $\mathbf{n}_e$ be the standard $\oK[t]$-basis of $N_F$ and $M_{C^{\otimes n}}\oplus M_{C^{\otimes n}}$ respectively. Since $\phi$ is $\oK[t]$-linear, we have $\phi(\mathbf{n}_F)=Q\mathbf{n}_e$ for some $Q\in\GL_{2}(\oK[t])$. Note that $\phi\circ\iota_1=\iota_2\circ\Id$ and $\Id\circ\pi_1=\pi_2\circ\phi$ implies that the matrix $Q$ must be of the form
        \[
            Q=\begin{pmatrix}
                1 & \\
                B & 1
            \end{pmatrix}
        \]
        for some $B\in\oK[t]$.
        
        In addition, we have
        \[
            \sigma \phi(\mathbf{n}_F)=\sigma Q\mathbf{n}_e=Q^{(-1)}\sigma\mathbf{n}_e=Q^{(-1)}\begin{pmatrix}
                (t-\theta)^n & \\
                 & (t-\theta)^n
            \end{pmatrix}\mathbf{n}_e.
        \]
        On the other hand,
        \[
            \phi(\sigma\mathbf{n}_F)=\phi(\Phi_F\mathbf{n}_F)=\Phi_F\phi(\mathbf{n}_F)=\begin{pmatrix}
                (t-\theta)^n & \\
                F & (t-\theta)^n
            \end{pmatrix}Q\mathbf{n}_e.
        \]
        Since $\phi$ is $\sigma$-equivariant, the above two equalities imply that
        \[
            Q^{(-1)}\begin{pmatrix}
                (t-\theta)^n & \\
                 & (t-\theta)^n
            \end{pmatrix}=\begin{pmatrix}
                (t-\theta)^n & \\
                F & (t-\theta)^n
            \end{pmatrix} Q.
        \]
        By comparing the entries of the above equation, 
        it satisfies the relation
        \begin{equation}
            B^{(-1)}(t-\theta)^n=F+B(t-\theta)^n
        \end{equation}
        In particular, we conclude that
        \[
            \Ker\delta\subset\{B^{(-1)}(t-\theta)^n-(t-\theta)^nB\mid B\in\oK[t]\}.
        \]
        The reverse inclusion can be obtained by a similar argument. The desired result now follows.
    \end{proof}
    
    An immediate consequence is the relations among $\overset{\rightarrow}{\mathscr{L}_{f,n}}$ arising from $\Ext^1_{\mathscr{F}}(M_{C^{\otimes n}},M_{C^{\otimes n}})$.
    \begin{corollary}\label{Cor:Same_Class_Relation}
        Let $f,g\in\oK[t]$. If $[N_{(-1)^nf}]=[N_{(-1)^n}g]$, then we must have $(-1)^nf-(-1)^ng\in(t-\theta)^n\oK[t]$. Moreover, for any
        \[
            B\in\wp^{-1}\bigg(\frac{(-1)^nf-(-1)^ng}{(t-\theta)^n}\bigg)\subset\oK[t],
        \]
        we have
        \[
            \overset{\rightarrow}{\mathscr{L}_{f,n}}=\overset{\rightarrow}{\mathscr{L}_{g,n}}+B.
        \]
    \end{corollary}

    \begin{proof}
        By Lemma~\ref{Lem:Ext_Identification_1}, the assumption $[N_{(-1)^nf}]=[N_{(-1)^ng}]$ implies that 
        \[
            (-1)^nf-(-1)^ng\in\{B^{(-1)}(t-\theta)^n-(t-\theta)^nB\mid B\in\oK[t]\}.
        \]
        Thus, for any
        \[
            B\in\wp^{-1}\bigg(\frac{(-1)^nf-(-1)^ng}{(t-\theta)^n}\bigg)\subset\oK[t]
        \]
        we have
        \[
            \wp\bigg(\overset{\rightarrow}{\mathscr{L}_{g,n}}+B\bigg)=\frac{(-1)^ng}{(t-\theta)^n}+\big(B^{(-1)}-B\big)=\frac{(-1)^nf}{(t-\theta)^n}
        \]
        which gives the desired result.
    \end{proof}
    
    For an integer $n\geq 1$, consider the $\oK$-vector space $\mathcal{V}_n:=\Mat_{n\times 1}(\oK)$ equipped with the $\mathbb{F}_q[t]$-module structure
    \begin{align*}
        \mathbb{F}_q[t]&\to\End_{\oK}(\mathcal{V}_n)\cong\Mat_{n\times n}(\oK)\\
        a&\mapsto [a]_n,
    \end{align*}
    that is determined by
    \[
        [t]_n=\begin{pmatrix}
                \theta & 1 & &\\
                 & \ddots & \ddots & \\
                 & & \ddots & 1\\
                 & & & \theta
            \end{pmatrix}.
    \]
    It has the natural sup-norm defined by
    \[
        |(x_{n-1},\dots,x_0)^\tr|_n:=\max_{0\leq i\leq n-1}\{|x_i|_\infty\}.
    \]
    The second characterization of the $\mathbb{F}_q[t]$-module $\Ext^1_{\mathscr{F}}(M_{C^{\otimes n}},M_{C^{\otimes n}})$ is given as follows.
    \begin{lemma}\label{Lem:TMod}
        We have an isomorphism of $\mathbb{F}_q[t]$-modules
        \[
            \Ext^1_{\mathscr{F}}(M_{C^{\otimes n}},M_{C^{\otimes n}})\cong\mathcal{V}_n.
        \]
    \end{lemma}

    \begin{proof}
        By Lemma~\ref{Lem:Ext_Identification_1}, we have an isomorphism between $\mathbb{F}_q[t]$-modules
        \[
            \Ext^1_{\mathscr{F}}(M_{C^{\otimes n}},M_{C^{\otimes n}})\cong\oK[t]/\{B^{(-1)}(t-\theta)^n-(t-\theta)^nB\mid B\in\oK[t]\}.
        \]
        We denote by $\mathcal{P}_n:=\{x_0+x_1(t-\theta)+\cdots+x_{n-1}(t-\theta)^{n-1}\mid x_0,\dots,x_{n-1}\in\oK\}\subset\oK[t]$. Then there is an $\mathbb{F}_q$-linear bijection between
        \[
            \oK[t]/\{B^{(-1)}(t-\theta)^n-(t-\theta)^nB\mid B\in\oK[t]\}\cong\mathcal{P}_n.
        \]
        It suffices to show that for any $F\in\oK[t]$, there is $G\in\oK[t]$ so that
        \[
            F-(G^{(-1)}-G)(t-\theta)^n=x_0+x_1(t-\theta)+\cdots+x_{n-1}(t-\theta)^{n-1}
        \]
        for some $x_0,\dots,x_{n-1}\in\oK$. The assertion is clear when $r:=\deg_tF<n$. Assume that $r\geq n$. We may express $F=F_0+F_1(t-\theta)+\cdots+F_r(t-\theta)^r$ and set $R_F:=F_n+F_{n+1}(t-\theta)+\cdots+F_r(t-\theta)^{r-n}$. By solving the Artin-Schreier equation coefficient-wise, there is $G\in\oK[t]$ with $G^{(-1)}-G=R_f$. Then one checks directly that
        \begin{align*}
            F-(G^{(-1)}-G)(t-\theta)^n&=F-R_F(t-\theta)^n\\
            &=F_0+F_1(t-\theta)+\cdots+F_{n-1}(t-\theta)^{n-1}
        \end{align*}
        This verifies the desired $\mathbb{F}_q$-linear bijection.

        Finally, we aim to show that the $\mathbb{F}_q$-linear bijection
        \begin{align*}
            \oK[t]/\{B^{(-1)}(t-\theta)^n-(t-\theta)B\mid B\in\oK[t]\}&\to\mathcal{V}_n\\
            \overline{x_0+x_1(t-\theta)+\cdots+x_{n-1}(t-\theta)^{n-1}}&\mapsto\begin{pmatrix}
                x_{n-1}\\
                \vdots\\
                x_1\\
                x_0
            \end{pmatrix}
        \end{align*}
        is an isomorphism of $\mathbb{F}_q[t]$-modules. Indeed, 
        \begin{align*}
            t\sum_{j=0}^{n-1}x_j(t-\theta)^j&=\big(\theta+(t-\theta)\big)\sum_{j=0}^{n-1}x_j(t-\theta)^j\\
            &=\theta x_0+\sum_{j=1}^{n-1}(\theta x_j+x_{j-1})+x_{n-1}(t-\theta)^n.
        \end{align*}
        If we simply choose $y$ to be a solution of the Artin-Schreier equation $X^{1/q}-X=x_{n-1}$, then 
        \[
            x_{n-1}(t-\theta)^n=(y^{(-1)}-y)(t-\theta)^n\in\{B^{(-1)}(t-\theta)^n-(t-\theta)B\mid B\in\oK[t]\}.
        \]
        It follows that the $t$-action
        \[
            t\cdot\sum_{j=0}^{n-1}\overline{x_j(t-\theta)^j}=\overline{\theta x_0}+\sum_{j=1}^{n-1}\overline{(\theta x_j+x_{j-1})}\in\{B^{(-1)}(t-\theta)^n-(t-\theta)B\mid B\in\oK[t]\}
        \]
        coincides with
        \[
            [t]_n\begin{pmatrix}
                x_{n-1}\\
                \vdots\\
                x_1\\
                x_0
            \end{pmatrix}=\begin{pmatrix}
                \theta & 1 & &\\
                 & \ddots & \ddots & \\
                 & & \ddots & 1\\
                 & & & \theta
            \end{pmatrix}\begin{pmatrix}
                x_{n-1}\\
                \vdots\\
                x_1\\
                x_0
            \end{pmatrix}.
        \]
        The desired result now follows.
    \end{proof}
    
    The main theme in this subsection is to give an explicit evaluation for our analytically continued $n$-th Kochubei polylogarithms and produce $\oK$-linear relations among them using the structure of $\Ext^1_{\mathscr{F}}(M_{C^{\otimes n}},M_{C^{\otimes n}})$. We begin with the following identity. Let $f\in\oK[t]$ with $\| f\|<q^n$. We express $f=f_0+f_1t+\cdots+f_mt^m$. Then we have
    \begin{equation}\label{Eq:Small_Deformation_Series}
        \mathscr{L}_{f,n}=\sum_{i\geq 1}\frac{f^{(i)}}{(\theta^{q^i}-t)^n}=\sum_{i\geq 1}\frac{f_0^{q^i}+f_1^{q^i}t+\cdots+f_m^{q^i}t^m}{(\theta^{q^i}-t)^n}=\sum_{j=0}^mt^j\mathscr{L}_{f_j,n}\in\TT_\theta.
    \end{equation}
    For each $n\in\mathbb{Z}_{>0}$, we denote by
    \[
        \mathcal{B}_n:=\{F\in\oK[t]\mid\|F\|<q^n\}\subset\oK[t]
    \]
    the set of small polynomials with respect to $n$. Our key result in this subsection is the generation of $\Ext^1_{\mathscr{F}}(M_{C^{\otimes n}},M_{C^{\otimes n}})$ by small polynomials over $\mathbb{F}_q[t]$.

    \begin{theorem}\label{Thm:Small_Gen}
        The following assertion holds.
        \[
            \Ext^1_{\mathscr{F}}(M_{C^{\otimes n}},M_{C^{\otimes n}})=\mathrm{Span}_{\mathbb{F}_q[t]}\{[N_G]\mid G\in\mathcal{B}_n\}.
        \]
        Consequently, given $f\in\oK[t]$, there is $g\in\mathcal{B}_n$, $\ell\in\mathbb{Z}_{\geq 0}$, and $B\in\oK[t]$ so that
        \[
            \overset{\rightarrow}{\mathscr{L}_{f,n}}=t^\ell\overset{\rightarrow}{\mathscr{L}_{g,n}}+B.
        \]
    \end{theorem}

    \begin{proof}
        Let $[N]\in\Ext^1_{\mathscr{F}}(M_{C^{\otimes n}},M_{C^{\otimes n}})$. By Lemma~\ref{Lem:Ext_Identification_1}, there is $F\in\oK[t]$ so that $[N]=[N_F]$. Moreover, using Lemma~\ref{Lem:TMod} there is $\bv_F=(x_{n-1},\dots,x_0)^\tr\in\mathcal{V}_n$ corresponds to $[N_F]$. Since
        \[
            [t]_n^{-1}=\begin{pmatrix}
                \theta & 1 & &\\
                 & \ddots & \ddots & \\
                 & & \ddots & 1\\
                 & & & \theta
            \end{pmatrix}^{-1}=\begin{pmatrix}
                \theta^{-1} & -\theta^{-2} & \cdots & (-1)^{1-n}\theta^{-n}\\
                 & \ddots & \ddots & \vdots\\
                 & & \ddots & -\theta^{-2}\\
                 & & & \theta^{-1}
            \end{pmatrix},
        \]
        we must have
        \[
            \lim_{\ell\to\infty}|[t]_n^{-\ell}\bv_F|_n=0.
        \]
        For each $\ell\geq 0$, we express $[t]_n^{-\ell}\bv_F=(x_{n-1}^{[\ell]},\dots,x_0^{[\ell]})^\tr$ and consider 
        \[
            G_\ell:=x_0^{[\ell]}+\cdots+x_{n-1}^{[\ell]}(t-\theta)^{n-1}\in\oK[t].
        \]
        Using $\|G_\ell\|_\theta=|[t]_n^{-\ell}\bv_F|_n$ and Lemma~\ref{Lem:CompareNorms}, there exists $\ell\geq 0$ such that $G_\ell\in\mathcal{B}_n$. Moreover, by the $\mathbb{F}_q[t]$-module isomorphism obtained in Lemma~\ref{Lem:TMod}, we conclude that
        \[
            [N_F]=t^\ell\cdot[N_{G_\ell}]=[N_{t^\ell G_\ell}]\in\mathrm{Span}_{\mathbb{F}_q[t]}\{[N_G]\mid G\in\mathcal{B}_n\}
        \]
        which implies the first assertion.

        To prove the second part, let $f\in\oK[t]$. Then there exist $g\in\mathcal{B}_n$ and $\ell\in\mathbb{Z}_{\geq 0}$ so that
        \[
            [N_{(-1)^nf}]=t^\ell\cdot[N_{(-1)^ng}]=[N_{t^\ell(-1)^ng}].
        \]
        It follows from the $\mathbb{F}_q[t]$-linearity of $\overset{\rightarrow}{\mathscr{L}_{-,n}}$ and Corollary~\ref{Cor:Same_Class_Relation} that there is $B\in\oK[t]$ with
        \[
            \overset{\rightarrow}{\mathscr{L}_{f,n}}=\overset{\rightarrow}{\mathscr{L}_{t^\ell g,n}}+B=t^\ell\overset{\rightarrow}{\mathscr{L}_{g,n}}+B
        \]
        which completes the proof.
    \end{proof}

    The above theorem provides an explicit evaluation of our analytically continued $n$-th Kochubei polylogarithms.
    \begin{corollary}\label{Cor:Series_Expression}
        Given $u\in\oK$, there exist $m\in\mathbb{Z}_{>0}$ and explicitly constructed $c\in\KK$, $a_1,\dots,a_m\in A$, as well as $f_1,\dots,f_m\in\KK$ with $|f_i|_\infty<q^n$ so that
        \[
            \overset{\rightarrow}{\Li_{\mathcal{K},n}}(u)=c+\sum_{j=1}^ma_j\overset{\rightarrow}{\Li_{\mathcal{K},n}}(f_j)=c+\sum_{j=1}^ma_j\left(\sum_{i\geq 1}\frac{f_j^{q^i}}{(\theta^{q^i}-\theta)^n}\right)\in\mathbb{C}_\infty/A.
        \]
    \end{corollary}

    \begin{proof}
        By specializing at $t=\theta$, the desired result follows immediately from \eqref{Eq:Small_Deformation_Series} and Theorem~\ref{Thm:Small_Gen}.
    \end{proof}

    In what follows, we provide an explicit example on how to produce $\oK$-linear relations among the $n$-th Kochubei polylogarithm from our machinery.
    
    \begin{example}
        Let $n=2$ and $u=\theta^2\in\oK$. Then we have the corresponding point $\bv=(0,\theta^2)^\tr\in\mathcal{V}_2$. We may compute
        \[
            [t]_2^{-1}\bv=\begin{pmatrix}
                \theta^{-1} & -\theta^{-2}\\
                 & \theta^{-1}
            \end{pmatrix}\begin{pmatrix}
                0\\
                \theta^2
            \end{pmatrix}=\begin{pmatrix}
                -1\\
                \theta
            \end{pmatrix}.
        \]
        The associated polynomial is given by $G_1=\theta-(t-\theta)=-t+2\theta\in\mathcal{B}_2$. It follows that
        \[
            [N_{\theta^2}]=t\cdot[N_{-t+2\theta}]=[N_{-t^2+2\theta t}]\in\Ext^1_{\mathscr{F}}(M_{C^{\otimes 2}},M_{C^{\otimes 2}}).
        \]
        One computes directly that
        \[
            -t^2+2\theta=t(-t+2\theta)=\big(\theta+(t-\theta)\big)\big(\theta-(t-\theta)\big)=\theta^2-(t-\theta)^2.
        \]
        We fix a solution of $X^{1/q}-X=-1$, denoted by $\epsilon\in\oK$. Then the relation obtained from Theorem~\ref{Thm:Small_Gen} reads as
        \[
            \overset{\rightarrow}{\mathscr{L}_{\theta^2,2}}=t\overset{\rightarrow}{\mathscr{L}_{-t+2\theta,2}}+\epsilon=t\big(\overset{\rightarrow}{\mathscr{L}_{2\theta,2}}+t\overset{\rightarrow}{\mathscr{L}_{-1,2}}\big)+\epsilon\in\TT_\theta/\mathbb{F}_q[t].
        \]
        By specializing at $t=\theta$, we get
        \[
            \overset{\rightarrow}{\Li_{\mathcal{K},2}}(\theta^2)=\epsilon+\theta\overset{\rightarrow}{\Li_{\mathcal{K},2}}(2\theta)-\theta^2\overset{\rightarrow}{\Li_{\mathcal{K},2}}(1)\in\mathbb{C}_\infty/A.
        \]
        In fact, one can realize this $\oK$-linear relation as the relation arising from the $\mathbb{F}_q[t]$-linear relation
        \[
            \begin{pmatrix}
                0\\
                \theta^2
            \end{pmatrix}=[t^2]_2\begin{pmatrix}
                0\\
                -1
            \end{pmatrix}+[t]_2\begin{pmatrix}
                0\\
                2\theta
            \end{pmatrix}.
        \]
        A refined version of this idea will be presented in the next subsection.
    \end{example}

\subsection{Linear relations among KPLs at algebraic points}
    In what follows, we aim to study linear relations among KPLs at algebraic points. The primary result in this subsection can be stated as follows.

    \begin{theorem}\label{Thm:Linear_Relations}
        For $n\geq 1$, let $u_1,\dots,u_\ell\in\oK$ with associated algebraic points $\bv_{u_1}=(0,\dots,0,(-1)^nu_1)^\tr,\dots,\bv_{u_\ell}=(0,\dots,0,(-1)^nu_\ell)^\tr\in\mathcal{V}_n$ in the $\mathbb{F}_q[t]$-module $\mathcal{V}_n$. If we fix an $\mathbb{F}_q$-linear lift $\Li_{\mathcal{K},n}^\circ:\mathbb{C}_\infty\to\mathbb{C}_\infty$, then we have the following inequalities.
        \begin{enumerate}
            \item $\dim_K\mathrm{Span}_K\{1,\Li_{\mathcal{K},n}^\circ(u_1),\dots,\Li_{\mathcal{K},n}^\circ(u_\ell)\}\geq\rank_{\mathbb{F}_q[t]}\mathrm{Span}_{\mathbb{F}_q[t]}\{\bv_{u_1},\dots,\bv_{u_\ell}\}$.
            \item $\rank_{\mathbb{F}_q[t]}\mathrm{Span}_{\mathbb{F}_q[t]}\{\bv_{u_1},\dots,\bv_{u_\ell}\}+1\geq\dim_{\oK}\mathrm{Span}_{\oK}\{1,\Li_{\mathcal{K},n}^\circ(u_1),\dots,\Li_{\mathcal{K},n}^\circ(u_\ell)\}$.
        \end{enumerate}
    \end{theorem}

    \begin{proof}
        To prove the first assertion, it suffices to show that any $K$-linear relation among $1,\Li_{\mathcal{K},n}^\circ(u_1),\dots,\Li_{\mathcal{K},n}^\circ(u_\ell)$ can be lifted to an $\mathbb{F}_q[t]$-linear relation among $\bv_{u_1},\dots,\bv_{u_\ell}$ in $\mathcal{V}_n$. Without loss of generality, assume that there are $a_0,a_1,\dots,a_\ell\in A$, not all zero, so that
        \[
            a_0+a_1\Li_{\mathcal{K},n}^\circ(u_1)+\cdots+a_\ell\Li_{\mathcal{K},n}^\circ(u_\ell)=0\in\mathbb{C}_\infty.
        \]
        For each $1\leq i\leq\ell$, we have
        \[
            \mathscr{L}_{u_i,n}^\circ\in\overset{\rightarrow}{\mathscr{L}_{u_i,n}}=\wp^{-1}\big(\frac{(-1)^nu_i}{(t-\theta)^n}\big)\in\TT_\theta
        \]
        so that $\mathscr{L}_{u_i,n}^\circ\mid_{t=\theta}=\Li_{\mathcal{K},n}^\circ(u_i)$. It follows that
        \[
            \bigg(a_0(t)+a_1(t)\mathscr{L}_{u_1,n}^\circ+\cdots+a_\ell(t)\mathscr{L}_{u_\ell,n}^\circ\bigg)\mid_{t=\theta}=0.
        \]
        We set
        \[
            \mathcal{R}:=a_0(t)+a_1(t)\mathscr{L}_{u_1,n}^\circ+\cdots+a_\ell(t)\mathscr{L}_{u_\ell,n}^\circ\in\TT_\theta
        \]
        and
        \[
            \mathcal{S}:=a_1(t)\frac{(-1)^nu_1}{(t-\theta)^n}+\cdots+a_\ell(t)\frac{(-1)^nu_\ell}{(t-\theta)^n}\in\oK(t).
        \]
        Using the relation $\mathcal{R}^{(-1)}-\mathcal{R}=\mathcal{S}$, one checks directly that
        \[
            \begin{pmatrix}
                \Omega^n \\
                \Omega^n\mathcal{R}
            \end{pmatrix}^{(-1)}=\begin{pmatrix}
                (t-\theta)^n & \\
                (t-\theta)^n\mathcal{S} & (t-\theta)^n
            \end{pmatrix}\begin{pmatrix}
                \Omega^n \\
                \Omega^n\mathcal{R}
            \end{pmatrix}.
        \]
        Since
        \[
            \begin{pmatrix}
                0 & 1
            \end{pmatrix}\begin{pmatrix}
                \Omega^n\\
                \Omega^n\mathcal{R}
            \end{pmatrix}\mid_{t=\theta}=0,
        \]
        by \cite[Thm.~3.1.1]{ABP04}, there exist $P_1,P_2\in\oK[t]$ with $P_1(\theta)=0$ and $P_2(\theta)=1$, such that
        \[
            \begin{pmatrix}
                P_1 & P_2
            \end{pmatrix}\begin{pmatrix}
                \Omega^n\\
                \Omega^n\mathcal{R}
            \end{pmatrix}=0.
        \]
        Since $P_2(\theta)=1$, we have $P_2\neq 0$. It follows that we can define $Q:=P_1/P_2\in\oK(t)$ and 
        \[
            \begin{pmatrix}
                Q & 1
            \end{pmatrix}\begin{pmatrix}
                \Omega^n\\
                \Omega^n\mathcal{R}
            \end{pmatrix}=0.
        \]
        
        On the one hand,
        \[
            \bigg(\begin{pmatrix}
                Q & 1
            \end{pmatrix}\begin{pmatrix}
                \Omega^n\\
                \Omega^n\mathcal{R}
            \end{pmatrix}\bigg)^{(-1)}=\begin{pmatrix}
                Q^{(-1)} & 1
            \end{pmatrix}\begin{pmatrix}
                (t-\theta)^n & \\
                (t-\theta)^n\mathcal{S} & (t-\theta)^n
            \end{pmatrix}\begin{pmatrix}
                \Omega^n\\
                \Omega^n\mathcal{R}
            \end{pmatrix}=0.
        \]
        On the other hand,
        \[
            (t-\theta)^n\begin{pmatrix}
                Q & 1
            \end{pmatrix}\begin{pmatrix}
                \Omega^n\\
                \Omega^n\mathcal{R}
            \end{pmatrix}=\begin{pmatrix}
                (t-\theta)^nQ & (t-\theta)^n            \end{pmatrix}\begin{pmatrix}
                \Omega^n\\
                \Omega^n\mathcal{R}
            \end{pmatrix}=0.
        \]
        By taking the difference of the above two equations, one deduces that
        \[
            \bigg(\begin{pmatrix}
                Q^{(-1)}(t-\theta)^n+(t-\theta)^n\mathcal{S}-(t-\theta)^nQ & 0
            \end{pmatrix}\bigg)\begin{pmatrix}
                \Omega^n\\
                \Omega^n\mathcal{R}
            \end{pmatrix}=0
        \]
        Since $\Omega^n\neq 0$, we must have
        \[
            Q^{(-1)}(t-\theta)^n+(t-\theta)^n\mathcal{S}-(t-\theta)^nQ=0.
        \]

        Let $N_{\mathcal{S}}$ be the Frobenius module defined by
        \[
            \Phi_{\mathcal{S}}:=\begin{pmatrix}
                (t-\theta)^n & \\
                (t-\theta)^n\mathcal{S} & (t-\theta)^n
            \end{pmatrix}\in\Mat_{2}(\oK[t]).
        \]
        Then $Q\in\oK(t)$ induces a $\oK(t)[\sigma]$-module homomorphism between $N_{\mathcal{S}}\otimes_{\oK[t]}\oK(t)$ and $\big(M_{C^{\otimes n}}\oplus M_{C^{\otimes n}}\big)\otimes_{\oK[t]}\oK(t)$ since we have the relation
        \[
            \begin{pmatrix}
                1 & \\
                Q & 1
            \end{pmatrix}^{(-1)}\begin{pmatrix}
                (t-\theta)^n & \\
                (t-\theta)^n\mathcal{S} & (t-\theta)^n
            \end{pmatrix}=\begin{pmatrix}
                (t-\theta)^n & \\
                 & (t-\theta)^n
            \end{pmatrix}\begin{pmatrix}
                1 & \\
                Q & 1
            \end{pmatrix}.
        \]
        It follows from \cite[Prop.~2.2.1]{CPY19} that there is $b\in\mathbb{F}_q[t]$ so that $bQ\in\oK[t]$. In particular, it shows that $b\cdot[N_\mathcal{S}]$ represents the trivial class in $\Ext^1_{\mathscr{F}}(M_{C^{\otimes n}},M_{C^{\otimes n}})$. By Lemma~\ref{Lem:TMod}, we conclude that
        \[
            [ba_1(t)]\bv_{u_1}+\cdots+[ba_\ell(t)]\bv_{u_\ell}=0\in\mathcal{V}_n,
        \]
        which implies the first inequality.

        To prove the second assertion, it suffices to show that any $\mathbb{F}_q[t]$-linear relation among algebraic points $\bv_{u_1},\dots,\bv_{u_{\ell}}\in\mathcal{V}_n$ can be lifted to an $\oK$-linear relation among
        \[
            1,\Li_{\mathcal{K},n}^\circ(u_1),\dots,\Li_{\mathcal{K},n}^\circ(u_\ell)\in\mathbb{C}_\infty.
        \]
        Without loss of generality, assume that there are $c_1,\dots,c_\ell\in\mathbb{F}_q[t]$, not all zero, so that
        \[
            [c_1]\bv_{u_1}+\cdots+[c_\ell]\bv_{u_\ell}=0\in\mathcal{V}_n.
        \]
        By Lemma~\ref{Lem:TMod}, the Frobenius module $N$ defined by
        \[
            \Phi:=\begin{pmatrix}
                (t-\theta)^n & \\
                c_1(-1)^nu_1+\cdots+c_\ell(-1)^nu_{\ell} & (t-\theta)^n
            \end{pmatrix}\in\Mat_{2}(\oK[t])
        \]
        represents the trivial class in $\Ext^1_{\mathscr{F}}(M_{C^{\otimes n}},M_{C^{\otimes n}})$. By Corollary~\ref{Cor:Same_Class_Relation} and the $\mathbb{F}_q[t]$-linearity of $\overset{\rightarrow}{\mathscr{L}_{-,n}}$, there is an explicitly constructed $B\in\oK[t]$ so that
        \[
            c_1\overset{\rightarrow}{\mathscr{L}_{u_1,n}}+\cdots+c_\ell\overset{\rightarrow}{\mathscr{L}_{u_\ell,n}}+B=0\in\TT_\theta/\mathbb{F}_q[t].
        \]
        In particular, there is $B'\in\oK[t]$ so that
        \[
            c_1\mathscr{L}_{u_1,n}^\circ+\cdots+c_\ell\mathscr{L}_{u_\ell,n}^\circ+B'=0\in\TT_\theta.
        \]
        By specializing at $t=\theta$, we deduce that
        \[
            c_1(\theta)\Li_{\mathcal{K},n}^\circ(u_1)+\cdots+c_\ell\Li_{\mathcal{K},n}^\circ(u_\ell)+B'(\theta)=0\in\mathbb{C}_\infty.
        \]
        This gives the desired $\oK$-linear relation which completes the proof.
    \end{proof}

    \begin{example}\label{Ex:Independence}
        Let $n=2$, $u_1=1$, and $u_2=\theta$. Then $u_1$ and $u_2$ are $K$-linearly dependent. Hence, \cite[Thm.~4.9]{Har22} (see also \cite[Rem.~4.10]{Har22}) has no conclusion in this situation. However, the existence of $a,b\in\mathbb{F}_q[t]$ so that
        \[
            [a]_2\begin{pmatrix}
                0\\
                1
            \end{pmatrix}+[b]_2\begin{pmatrix}
                0\\
                \theta
            \end{pmatrix}=0\in\mathcal{V}_2
        \]
        is equivalent to
        \[
            \begin{cases}
                a(\theta)'+b(\theta)'\theta=0\\
                a(\theta)+b(\theta)\theta=0
            \end{cases},
        \]
        where $a(\theta)'$ and $b(\theta)'$ refer to the derivative of $a(\theta)$ and $b(\theta)$ respectively. One verifies easily that the only possible solution is $a=b=0$. Hence
        \[
            \rank_{\mathbb{F}_q[t]}\mathrm{Span}_{\mathbb{F}_q[t]}\{(0,1)^\tr,(0,\theta)^\tr\}=2.
        \]
        By Theorem~\ref{Thm:Linear_Relations}, we conclude that
        \[
            \dim_K\mathrm{Span}_K\{1,\Li_{\mathcal{K},2}^\circ(1),\Li_{\mathcal{K},2}^\circ(\theta)\}=3.
        \]
        Moreover, since the original power series expression $\Li_{\mathcal{K},2}(z)$ converges at both $1$ and $\theta$, we further conclude that any non-trivial $K$-linear combination of $\Li_{\mathcal{K},2}(1)$ and $\Li_{\mathcal{K},2}(\theta)$ is irrational in the following sense. For any $c_1,c_2\in K$, not all zero, we have
        \[
            c_1\Li_{\mathcal{K},2}(1)+c_2\Li_{\mathcal{K},2}(\theta)\not\in K.
        \]
    \end{example}

\end{document}